\documentclass[11pt,letterpaper]{amsart}

\usepackage[utf8]{inputenc}
\usepackage{amsmath}
\usepackage{amsfonts}
\usepackage{amssymb}
\usepackage{mathrsfs}
\usepackage{graphicx}

\usepackage[left=3cm,right=3cm, top=3cm,bottom=3cm]{geometry}

\usepackage[colorlinks=true, linkcolor=blue, citecolor=blue, urlcolor=blue]{hyperref}

\usepackage[capitalise]{cleveref}
\usepackage{longtable}

\usepackage{tikz-cd} 
\usepackage[all,cmtip]{xy}	
\xyoption{arrow}
\usepackage{dynkin-diagrams}

\usepackage{xcolor,etoolbox}
\usepackage{enumitem}
\usepackage{adjustbox}

\usepackage[textsize=footnotesize, textwidth=25mm, color=green!40]{todonotes}

\usepackage{array}
\usepackage{multirow}

\theoremstyle{plain}

\newtheorem{theorem}[subsubsection]{Theorem}
\newtheorem*{thm:H1G}{\Cref{thm:H1tauG}}
\newtheorem*{thm:H622}{\Cref{thm:parahoricfromcoverings}}
\newtheorem*{thm:412421}{\Cref{thm:local} and \Cref{thm:BT}}
\newtheorem*{cor421}{\Cref{thm:BT}}

\newtheorem{lemma}[subsubsection]{Lemma}
\newtheorem{prop}[subsubsection]{Proposition}
\newtheorem{cor}[subsubsection]{Corollary}

\theoremstyle{definition}
\newtheorem{defi}[subsubsection]{Definition}

\theoremstyle{remark}
\newtheorem{rmk}[subsubsection]{Remark}
\newtheorem{eg}[subsubsection]{Example}

\newcommand{\NN}{\mathbb{N}}
\newcommand{\ZZ}{\mathbb{Z}}
\newcommand{\QQ}{\mathbb{Q}}
\newcommand{\RR}{\mathbb{R}}
\newcommand{\CC}{\mathbb{C}}

\newcommand{\Ho}{\mathrm{H}}

\newcommand{\car}{\mathrm{char}}

\newcommand{\Uc}{\mathcal{U}}
\newcommand{\Gc}{\mathcal{G}}
\newcommand{\Tc}{\mathcal{T}}
\newcommand{\Rc}{\mathcal{R}}
\newcommand{\Sc}{\mathcal{S}}

\newcommand{\Ac}{\mathcal{A}}
\newcommand{\Ec}{\mathcal{E}}
\newcommand{\Fc}{\mathcal{F}}
\newcommand{\Hc}{\mathcal{H}}
\newcommand{\Pc}{\mathcal{P}}
\newcommand{\Eco}{\mathcal{E}^\circ}
\newcommand{\Kc}{\mathcal{K}}
\newcommand{\Oc}{\mathcal{O}}
\newcommand{\Arm}{\mathcal{A}_{r/m}}
\newcommand{\Srm}{\mathcal{S}_{r/m}}

\newcommand{\Av}{\mathsf{Av}}
\newcommand{\norm}{\mathsf{N}}
\newcommand{\coX}{X_*}
\newcommand{\coP}{\check{P}}
\newcommand{\Db}{\mathbb{D}}

\newcommand{\XC}{X}
\newcommand{\qp}{f}
\newcommand{\QP}{q}
\newcommand{\Ct}{\widetilde{C}}
\newcommand{\calpha}{\check{\alpha}}
\newcommand{\cbeta}{\check{\beta}}
\newcommand{\cvarpi}{\check{\varpi}}

\newcommand{\Ad}{\mathrm{Ad}}
\newcommand{\ad}{\mathrm{ad}}
\newcommand{\Spec}{\mathrm{Spec}}
\newcommand{\Res}{\mathrm{Res}}
\newcommand{\Bun}{\mathrm{Bun}}
\newcommand{\Aut}{\mathrm{Aut}}
\newcommand{\Out}{\mathrm{Out}}

\newcommand{\Ga}{\mathbb{G}_{\mathrm{a}}}
\newcommand{\Gm}{\mathbb{G}_{\mathrm{m}}}
\newcommand{\SL}{\textrm{SL}}
\newcommand{\PGL}{\textrm{PGL}}
\newcommand{\GL}{\textrm{GL}}
\newcommand{\Iso}{\mathcal{I}\textit{so}}

\newcommand{\g}{\mathfrak{g}}

\renewcommand{\arraystretch}{1.5}
\usepackage[justification=centering]{caption}

\author[Chiara Damiolini]{Chiara Damiolini}
\address{\textrm{Chiara Damiolini} \newline \indent Department of Mathematics, University of Texas at Austin, Austin, TX}
\email{chiara.damiolini@austin.utexas.edu}

\author[Jiuzu Hong]{Jiuzu Hong}
\address{\textrm{Jiuzu Hong} \newline \indent Department of Mathematics, University of North Carolina at Chapel Hill, Chapell Hill, NC}
\email{jiuzu@email.unc.edu}

\title[Local types of $(\Gamma,G)$-bundles and parahoric group schemes]{Local types of $(\Gamma,G)$-bundles\\ and parahoric group schemes}

\begin{document}

\begin{abstract}
Let $G$ be a simple algebraic group over an algebraically closed field $k$. Let $\Gamma$ be a finite group acting on $G$. We classify and compute the local types of $(\Gamma, G)$-bundles on a smooth projective $\Gamma$-curve in terms of the first non-abelian group cohomology of the stabilizer groups at the tamely ramified points with coefficients in $G$. When $\text{char}(k)=0$, we prove that any generically simply-connected parahoric Bruhat--Tits group scheme can arise from a $(\Gamma,G_{\ad})$-bundle. We also prove a local version of this theorem, i.e.  parahoric group schemes over the formal disc arise from constant group schemes via tamely ramified coverings.
\end{abstract}
\maketitle

\section{Introduction}

Moduli spaces of vector bundles on a smooth and projective curve $C$ have been a central object in algebraic geometry, with tight relations to representation theory and also conformal field theory. A natural generalization of these objects, which encompass parabolic bundles or prym varieties, is to instead consider principal $\Gc$-bundles over $C$, where $\Gc$ is a \textit{parahoric Bruhat--Tits group}. The definition of this type of group can be roughly described by saying that the generic fiber of $\Gc$ is reductive, it has connected geometric fibers, and for every point $x \in C$ such that $\Gc|_x$ is not reductive, the sections of $\Gc$ over a formal disk about $x$ defines a parahoric group in the sense of \cite{bruhat.tits:1984:II}, see \cref{def:BT}.  The points $x$ with non reductive fiber $\Gc|_x$ are called ramified points of $\Gc$. In what follows we will assume further that the generic fiber of $\Gc$ is actually simple and not merely reductive. One of the aims of this paper, developed in \cref{sec:BTfromcoverings} and inspired by \cite{balaji.seshadri:2015:moduli}, is to provide a direct description of these groups through (ramified) Galois coverings and the concept of $(\Gamma,G)$-bundles. 

Parahoric Bruhat--Tits groups have been introduced in \cite{pappas.rapoport:2010:questions}, where the authors outline a series of conjectures about the moduli space $\Bun_\Gc$ of $\Gc$-bundles. Among others, they suggest the description of spaces of generalized theta functions on $\Bun_\Gc$ via an appropriate notion of conformal blocks, which would be the first step to obtain a Verlinde-type formula to compute their dimension. Motivated by this paper and by the uniformization theorem for $\Bun_\Gc$ \cite{heinloth:2010:uniformization}---conjectured as well in \cite{pappas.rapoport:2010:questions}---many mathematicians have worked on twisted conformal blocks \cite{damiolini:2020:conformal,zelaci:2019,Hong-Kumar:2019,deshpande.mukhopadhyay:2019,hong.kumar:2022}. These are a natural generalization of the conformal blocks attached to a curve $C$ and to representations of simple Lie algebras (see, e.g., \cite{tuy}), where the simple Lie algebra is replaced by a pair $(\Gamma,\g)$ consisting of a simple Lie algebra $\g$ and a finite group $\Gamma$ acting on $\g$. Its representation theory is described by (possibly twisted) affine Lie algebras and, under appropriate conditions, one can identify them with appropriate spaces of generalized theta functions on $\Bun_{\Gc}$, where $\Gc$ is a parahoric Bruhat--Tits group on $C$ arising from a $\Gamma$ covering $\Ct$ of $C$ \cite{Hong-Kumar:2019}.

As we have mentioned, one of the goals is to express parahoric Bruhat--Tits groups through $(\Gamma,G)$-bundles, concept introduced in \cite{balaji.seshadri:2015:moduli} and then extended by the first author in \cite{damiolini:2021:equivariant}. Let $\Gamma$ be a finite group and assume that it acts on a smooth and projective curve $\Ct$ over $\Spec(k)$ so that $\pi \colon \Ct \to \Ct/\Gamma =: C$ is a possibly ramified covering of the curve $C$ (we assume here that $\car(k)$ does not divide the order of $\Gamma$). Let $G$ be a simple algebraic group over $k$ and assume that $\Gamma$ acts on $G$ as well, so that one obtains an induced action of $\Gamma$ on $G \times \Ct$. Under these assumptions, the group $\Gc:= \pi_*(G \times \Ct)^\Gamma$ is a parahoric Bruhat--Tits group over $C$. We can generalize this construction, by considering $(\Gamma,G_\ad)$-bundles, i.e.  principal $G_\ad$-bundles over $\Ct$ which are further equipped with a compatible action of $\Gamma$. If $\Ec$ is a $(\Gamma,G_\ad)$-bundle, then the group scheme $G_\Ec:= \Ec \times^{G_\ad} G$ is still equipped with an action of $\Gamma$ lifting the action on $\Ct$, so that  $\Gc_\Ec:= \pi_*(G_\Ec)^\Gamma$ defines a smooth group scheme over $C$. This is actually a parahoric Bruhat--Tits group over $C$ and, according to the following theorem, this construction essentially recovers all parahoric Bruhat--Tits groups. 

\begin{thm:H622} Let $\Gc$ be a parahoric Bruhat--Tits group over $C$, a smooth and projective curve of genus $g \geq 1$ over an algebraically closed field $k$ of characteristic zero. Assume that $\Gc$ is generically simple and simply connected, and let $G$ be the group over $k$ with the same root datum as $\Gc_{k(C)}$, where $k(C)$ is the function field of $C$. Then there exists a finite group $\Gamma$, a $\Gamma$-covering $\pi \colon \Ct \to C$, and a $(\Gamma,G_{\ad})$-bundle $\Ec$ on $\Ct$ such that $\Gc \cong \pi_*(G_\Ec)^\Gamma$.
\end{thm:H622}

This result is stated in slightly more generality in \cref{thm:parahoricfromcoverings}, where curves of genus zero are allowed, provided that $\Gc$ is ramified at at least two points, or not ramified at all.  This result can be seen as a generalization of \cite[Theorem 5.2.7]{balaji.seshadri:2015:moduli}: in fact, Balaji and Seshadri assume that the parahoric Bruhat--Tits group $\Gc$ is generically split, while we do not restrict to this case. 

The proof of  \cref{thm:parahoricfromcoverings} relies on its local counterpart, developed in \cref{sec:local} and which we summarize here. For every integer $N$ will use the notation $\Kc_N$ (resp. $\Oc_N$) to denote $k(\!(z^{1/N})\!)$ (resp. $k[\![z^{1/N}]\!]$). Moreover, if $\sigma$ is an automorphism of $G$ with $\sigma^N=1$, we let $\sigma$ act on $\Kc_N$ as well by $\sigma(z^{1/N})=\zeta_N^{-1}z^{1/N}$ for a primitive $N$-th root of unity $\zeta_N$.

\begin{cor421} Let $\tau$ be a diagram automorphism of order $r$ of the simple and simply connected group $G$, and let $T$ be a maximal torus preserved by $\tau$.  Let $\Gc_\theta$ be a parahoric group scheme of the twisted loop group $G(\Kc_r)^\tau$ corresponding to a point $\theta$ in $X_*(T)^\tau_{\QQ}$. Let $m$ be the minimal positive integer such that $m\theta$ is in the lattice $X_*(T_\ad)^\tau$, where $T_{\ad}$ is the image of $T$ in $G_{\ad}$, the adjoint quotient of $G$. If $\car(k)$ does not divide $m$, then 
\[\Gc_\theta \cong {\Res}_{\Oc_m/\Oc}\left(G_{\Oc_m}\right)^\sigma,\] for some automorphism $\sigma$ on $G$ such that $\sigma^m=1$, and $\bar{\sigma}=\bar{\tau} \in {\rm Out}(G)$.
\end{cor421}

This can be viewed as an extension of \cite[Theorem 2.3.1]{balaji.seshadri:2015:moduli}, where the authors assumed $\tau$ is trivial. Here instead, $\sigma$ can be any finite order automorphism of $G$, see \cref{cor:parahoric}. 
Moreover, we work over a field $k$ of possibly positive characteristic, while \cite{balaji.seshadri:2015:moduli} considers only the case $k=\CC$. In \cref{rmk:adjoint} we also note that the above result remains true if $r=2$ (resp. $r=3$) and $G$ is adjoint of type $A_{2\ell}$, $E_6$ (resp. $D_4$). 

The key step towards proving this result is \cref{thm:local}, where we show that indeed the group scheme ${\Res}_{\Oc_m/\Oc}\left(G_{\Oc_m}\right)^\sigma$ is a parahoric groups scheme $\Gc_\theta$ for an appropriate $\theta$ which we can determine explicitly using \cref{thm:H1tauG}. This heavily relies on \cref{prop:kacgroup}, which can be seen as the group theoretic counterpart of the analogous isomorphism that holds at the level of affine Lie algebras \cite[Theorem 8.5]{kac:1990:infinite}. This perspective will also be the underlying point of view of \cref{appendix} and in particular of \cref{thm:clas}.

\medskip 

The principal consequence of \cref{thm:parahoricfromcoverings} is that the study of parahoric Bruhat--Tits group can be translated into the analysis of $(\Gamma,G)$-bundles. 
We first of all show that any two $(\Gamma,G)$-bundles are isomorphic outside the ramification points of the covering $\Ct \to C$. We then call the isomorphism class of a $(\Gamma,G)$-bundle around a ramification point the \textit{local type} of the bundle (as in \cite{balaji.seshadri:2015:moduli} and \cite{damiolini:2021:equivariant}).  
We further give a combinatorial description of the local types using the \textit{non-abelian cohomology} $\Ho^1(\Gamma_x,G)$ for the cyclic group $\Gamma_x \subset \Gamma$ fixing the ramified point $x$. Intuitively, if we see the $(\Gamma,G)$-bundle $\Ec$ only as a $G$-bundle, this is trivial on the formal neighborhood $\Db_x$ of $x$. It is the action of $\Gamma_x$ on this trivial $G$-bundle that can be explicitly determined by an element $G(\Db_x)$ satisfying appropriate cocycle conditions.
By changing the trivialization, this gives rise to an element of $\Ho^1(\Gamma_x,G(\Db_x))$, which in turns coincides with $\Ho^1(\Gamma_x,G)$ (see  \cref{prop:twgen} as well as \cite[Lemma 2.5]{teleman.woodward:2003:parabolic}). In \cref{thm_comp} we further show that, when $G$ is simply-connected, the moduli stack $\Bun_{\Gamma,G,\vec{\kappa}}$ parametrizing $(\Gamma,G)$-bundles of local type $\vec{\kappa}$ at ramified points is isomorphic to $\Bun_{\Gc}$ for some parahoric Bruhat--Tits group $\Gc$.

\medskip

We are then left to give an explicit description of $\Ho^1(\Gamma,G)$ for a cyclic group $\Gamma$, whose generator $\gamma$ acts on the simple group $G$ via an automorphism of $G$. Resorting to previous works on \textit{twisted conjugacy classes} for simple algebraic groups and related invariant theory \cite{mohrdieck:2003:conjugacy, springer:2006:twisted}, we can explicitly describe $\Ho^1(\Gamma,G)$ when $\gamma$ acts as a diagram automorphism $\tau$ of $G$ as follows. 

\begin{thm:H1G} Let $m$ be the order of $\Gamma$ and let its generator $\gamma$ act on $G$ by the diagram automorphism $\tau$ of order $r$. Assume that  $\car(k)$ does not divide $m$. Then if either $k = \CC$, or $G$ is simply connected or adjoint, then 
\[ \Ho^1(\Gamma,G) \cong \dfrac{\frac{1}{m}\coX(T)^\tau}{\Av_\tau(\coX(T)) \rtimes W^\tau},
\] where $\coX(T)$ denotes the set of cocharacters of $T$, by $\Av_\tau$ we denote the map averaging the action of $\tau$, and $W$ is the Weyl group of $G$.
\end{thm:H1G}

Using this explicit description, we can compare the chamber of Bruhat--Tits buildings for $G(\Kc_r)^\tau$ with the possible local types of $(\Gamma,G)$-bundles. This is used not only as an effective way to compute local types of $(\Gamma,G)$-bundles, but also as an ingredient used to prove the results of \cref{sec:local}.

In \cref{appendix} we explore further consequences of \cref{thm:H1tauG}. When $G$ is simply connected,  this can be identified with certain rational points in the fundamental alcove of the affine Weyl group $\Av_\tau(\coX(T)) \rtimes W^\tau$, see \cref{lemma:alcove}. When $G$ is adjoint,  $\Ho^1(\Gamma,G)$ has an even more concrete realization, as described in \cref{prop:alcove}. This has as a consequence that we can extend Kac's classification of automorphisms of $G$ to fields of positive characteristic not dividing the order of $\Gamma$ (\cref{thm:clas}).

\medskip Some of the results obtained in this paper have been independently obtained by Pappas and Rapoport in the recent preprint \cite{Pappas-Rapoport:2022}. Their work is more arithmetic oriented than the one presented here. For instance, their description of local types are Bruhats-Tits theoretic, while our description is more group-theoretical and combinatorial in nature. One can also compare  \cite[Theorem 1.1]{Pappas-Rapoport:2022} and \cref{thm:parahoricfromcoverings}: under different assumptions---we restrict to characteristic zero---both statements assert that parahoric Bruhat--Tits group schemes arise from coverings, but the methods used to reach this result are different.  Our result is based on its local counterpart (\cref{thm:local}), which we prove in a group-theoretical way and where we do not require characteristic zero. In a previous version of our paper, the argument for \cref{cor:comp} was incomplete. Since this is one of the main results of \cite{Pappas-Rapoport:2022} and it does not represent a key element of this paper, in this version of the paper we use their work to deduce this result.

\subsection*{Structure of the paper} In \cref{sec:parahoric} we recall the definition of parahoric groups and of parahoric Bruhat--Tits group schemes. In \cref{sec:non-abelian} we introduce the concept of non-abelian group cohomology. We give an explicit description of $\Ho^1(\Gamma,G)$ when the cyclic group $\Gamma$ acts on a simple group $G$ by diagram automorphisms in \cref{thm:H1tauG}. In \cref{sec:local} we give an explicit description of the local structure of parahoric group schemes (see \cref{thm:local}). In order to extend this description to a global one, we introduce in \cref{sec:localtypes} the concept of $(\Gamma,G)$-bundle and show that local types uniquely identify them (\cref{prop:nonempty}). Finally, in \cref{sec:BTfromcoverings}, we show that all parahoric Bruhat--Tits groups over a projective smooth curve can be realized from $(\Gamma,G_\ad)$-bundles (\cref{thm:parahoricfromcoverings}). Besides providing many concrete examples of local types and their computations, in \cref{appendix} we also describe $\Ho^1(\Gamma,G)$ in terms of rational points of  the fundamental alcove and classify finite order automorphisms of simple algebraic groups (\cref{thm:clas}). 

\medskip

\subsection*{Acknowledgements} The authors wish to thank J. Heinloth for conversations had during the conference ``Bundles and Conformal Blocks with a Twist'' held at ICMS, Edinburgh. Many thanks also to S. Kumar,  G. Pappas and M. Rapoport for their interest in this work and their comments on a previous version of this paper. We also thank Cheng Su for communicating to us some similar ideas on local types that are presented in this work. J. Hong is grateful to R. Travkin for many stimulating discussions on the global aspects of parahoric Bruhat--Tits group schemes in 2021. Many thanks to the anonymous referee, whose helpful comments significantly improved this paper.  J. Hong was partially supported by the NSF grant DMS-2001365.

\subsection*{Notation}

Throughout we work over an algebraically closed field $k$ of characteristic $p$. Assumptions on the characteristic will vary throughout the paper, and will always be allowed to have characteristic zero.

We denote the valuation field $k(\!(z)\!)$ by $\Kc$ and its ring of integers $k[\![z]\!]$ by $\Oc$.  For every $r \in \NN$ such that $p\not \vert r$, we set $\Kc_r:=  k(\!(z^{\frac{1}{r}})\!)$ and similarly $\Oc_r:=k[\![z^{\frac{1}{r}}]\!]$. The discrete valuation $\mathfrak{v}$ of $\Kc$ extends to a unique valuation $\mathfrak{v}_r$ of $\Kc_r$ with values in $\frac{1}{r}\ZZ$ such that $\mathfrak{v}_rz^{\frac{1}{r}} = \frac{1}{r}$.

We fix a generator $\gamma$ of $\text{Gal}(\Kc_r/\Kc)$ and assume that a primitive $r$-th root of unity $\zeta_r \in k$ is fixed so that $\gamma(z^{\frac{1}{r}}) := \zeta_r^{-1}z^{\frac{1}{r}}$. 

Geometrically, this situation corresponds to the tamely ramified Galois covering between formal disks
\[ \pi \colon  \Db_r:=\Spec(\Oc_r) \longrightarrow \Db:=\Spec(\Oc),
\]which restricts to an étale covering $\Db_r^\times :=\Spec(\Kc_r) \to \Db^\times:=\Spec(\Kc)$ between punctured disks. Similarly, when $x \in C$ is a smooth point of a reduced curve, we will denote by $\Oc_x$ the complete local ring at $x$, in formulas $\Oc_x= \varprojlim \mathscr{O}_{C,p}/\mathfrak{m}_x^n$, and $\Db_x:=\Spec(\Oc_x)$. By choosing a local coordinate $z$ at $x$, these are isomorphic to $\Oc$ and $\Db$. One analogously defines the field $\Kc_x$ and  the scheme $\Db_x^\times=\Spec(\Kc_x)$.

For every scheme $X$ over $\Db^\times$ (or $\Db$), we will denote by $X_{\Kc_r}$ (resp $X_{\Oc_r}$) the fiber product $X \times_{\Db^\times} \Db_r$ (resp. $X \times_\Db \Db_r$).

\section{Parahoric Bruhat--Tits group schemes} \label{sec:parahoric}
In this introductory section, we recall the notions of parahoric groups and the associated group schemes over $\mathcal{O}$. Following \cite{heinloth:2010:uniformization}, in the second half of this section we introduce a \textit{global} version of these groups, that is parahoric Bruhat--Tits group schemes over curves. 

\subsection{Parahoric group schemes} We briefly present here the notion of parahoric group schemes, introduced in \cite{bruhat.tits:1984:II} and to which we refer for more details. These are connected group schemes $\Gc$ over $\Oc$ such that $\Gc_{\Kc}$ are reductive and the sections $\Gc(\Oc)$ define a \textit{parahoric subgroup} $\Pc$ of $\Gc(\Kc)$. They are a generalization of \textit{constant} group schemes $G_{\Oc}$ for a reductive group scheme $G$ over $k$. In this paper we only consider the case in which $\Gc_{\Kc}$ is required to be absolutely simple, and we begin describing such groups.

In this section we denote by $\Gc$ a quasi-split absolutely simple group scheme over $\Kc$ which splits over a tamely ramified extension. This implies that we can write \[\Gc \cong \Res_{\Kc_r/\Kc}(G_{\Kc_r})^\tau,\]
where $G$ is a simple algebraic group over $k$, $\tau$ is a diagram automorphism of order $r$, and $\car(k)$ does not divide $r$. The action of $\tau$ on $\Kc_r$ is $k$-linear and determined by $\tau(z^{\frac{1}{r}}) = \zeta_r^{-1} z^{\frac{1}{r}}$ for $\zeta_r$ a primitive $r$-th root of unity.

The maximal torus $\Tc$ of $\Gc$ can be analogously described as $\Res_{\Kc_r/\Kc}(T_{\Kc_r})^\tau$ and contains the maximal split torus $\Sc$ of $\Gc$ which is isomorphic to the connected component of $(T_{\Kc})^\tau$, which we denote $T_\Kc^{\tau,\circ}$. The relative root system $\Rc$ for $(\Gc, \mathcal{S})$ (which is not necessarily reduced) can be seen as a quotient of the root system $R$ of $(G,T)$. We call an element $a$ of $\Rc$ a \textit{multiple or divisible} root if $2a$ or $\frac{1}{2}a$ is also an element of $\Rc$. This case can only happen when $\tau$ has order two and switches adjacent roots in the Dynkin diagram of $G$. In order to describe parahoric subgroups of $\mathcal{G}(\Kc)$, we briefly introduce the root subgroups $\Uc_{a}$ of $\mathcal{G}$ for every $a\in \Rc$. In what follows, we denote by $\Delta_a \subset R$ the set of preimages of $a \in \Rc$. 

When $a$ is not multiple or divisible, then the group $\Uc_a$ is given by 
\[\left(\prod_{\alpha \in \Delta_a} U_\alpha(\Kc_\alpha)\right)^{\tau},\] 
where $\Kc_\alpha$ is the subfield of $\Kc_r$ which is fixed by the stabilizer of any $\alpha \in \Delta_a$, in formulas $\Kc_\alpha= \Kc_r^{\text{Stab}_{\alpha}}$. Since the action of $\tau$ on $\Delta_a$ is transitive, one can identify $\Uc_a$ with $U_\alpha(\Kc_\alpha)$. In this way, the map $x_\alpha \colon \Ga(\Kc_\alpha) \to U_\alpha(\Kc_\alpha)$ induces the Chevalley-Steinberg pinning $x_a \colon \Ga(\Kc_\alpha) \to \Uc_a$. 

\begin{eg} In the case of $\Delta_a=\{\alpha, \tau(\alpha)\}$, then 
\[\Uc_a=\{ A \tau(A) \quad \text{with} \quad A \in U_\alpha(\Kc_2)\}.
\] Note that $\tau(A) \in U_{\tau(\alpha)}(\Kc_2)$. The identification of $U_\alpha(\Kc_2)$ with $\Uc_a$ is given by $A \mapsto A \tau(A)$. \end{eg}

The discrete valuation $\mathfrak{v}$ of $\Kc$ induces a unique valuation $\mathfrak{v} \colon \Kc_\alpha \to \frac{1}{r}\ZZ$ and, through the pinning $x_\alpha$, this induces a valuation $\mathfrak{v}_a$ on non identity elements of $\Uc_a$. Namely, for every $A \in \Uc_a \setminus \{0\}$ we set $\mathfrak{v}_a(A) := \mathfrak{v}(x_a^{-1}(A))$. This allows us to define, for every $\ell \in \RR$, the group $\Uc_{a,\ell} := \mathfrak{v}_a^{-1}[\ell, \infty) \subseteq \Uc_a$.  

One similarly defines $\Uc_a$ and $\Uc_{a,\ell}$ for multiple and divisible roots, and we refer the reader to \cite[4.1.9]{bruhat.tits:1984:II} and \cite[Appendix]{heinloth:2017:stability}.

Consider the affine space generated by the cocharacters of $\Sc$, i.e. $E=X_*(\Sc) \otimes_\ZZ \RR$, which we can identify with $X_*(T^\tau) \otimes_\ZZ \RR$. The bilinear map $X_*(\Sc) \times X^*(\Sc) \to \ZZ$  induces the map \[E \times \Rc \to \RR,  \qquad (\theta, a) \mapsto \theta(a).\]

\begin{defi} \label{def:parahoric} We define the parahoric group associated with $\theta \in E$ to be the subgroup $\Pc_\theta$ of $\Gc(\Kc)$ which is generated by $T(\Oc_r)^{\tau,0}$  and $\lbrace \Uc_{a,-\theta(a)}  \rbrace_{a \in \Rc}$.  By \cite[3.8.1, 3.8.3, 4.6.2, 4.6.26, 5.2.6]{bruhat.tits:1972:I,bruhat.tits:1984:II} there exists a unique affine smooth group scheme $\Gc_{\theta}$ over $\Oc$ which extends $\Gc$ and such that $\Gc_\theta(\Oc) = \Pc_{\theta}$. The group $\Gc_{\theta}$ is called the \textit{parahoric group scheme} associated with $\theta$. \end{defi}

\subsection{Parahoric Bruhat--Tits group schemes} In the same spirit in which $\Db^\times=\Spec(\Kc)$ and $\Db=\Spec(\Oc)$ detect the local behavior of a smooth curve at a point, the group schemes that we have described in the previous section provide the local description of parahoric Bruhat--Tits group schemes. More precisely, we have the following definition:

\begin{defi} \label{def:BT} Let $C$ be a smooth projective curve over $k$. We say that a smooth affine group scheme $\Gc$ over $C$ is a \textit{parahoric Bruhat--Tits group scheme} over $C$ if it satisfies the following conditions:
\begin{enumerate}
\item All geometric fibers of $\Gc$ are connected.
\item \label{item:defBT2} The generic fiber of $\Gc$ is simple.
\item For all $x \in C$ such that the fiber $\Gc|_x$ is not simple, the group scheme $\Gc|_{\Db_x}$ is a parahoric group scheme extending $\Gc|_{\Db_x^\times}$, that is $\Gc(\Db_x)$ is a parahoric subgroup of $\Gc(\Db_x^\times)$.\end{enumerate}\end{defi}

In \cref{sec:BTfromcoverings} we will see how, under appropriate conditions, parahoric Bruhat--Tits groups can be recovered from coverings of curves through the concept of $(\Gamma,G)$-bundles. 

\begin{rmk} The definitions of parahoric Bruhat--Tits groups given in \cite{heinloth:2010:uniformization,balaji.seshadri:2015:moduli}, allows the generic fiber to be semisimple, rather than simple as in our definition. This simplifies our arguments, but it is not a very restrictive condition.
However, we note that in \cite[Definition 5.2.1]{balaji.seshadri:2015:moduli} condition (\ref{item:defBT2}) is replaced by the assumption that there exists a Zariski open subset $U \subseteq C$ such that $\Gc|_U \cong G \times U$, for $G$ a simple group scheme over $k$. From \cref{prop:genericallyreductive}, in fact this is equivalent to require the parahoric Bruhat--Tits group scheme defined in \cref{def:BT} to be generically split.
\end{rmk}

\section{Non-abelian group cohomology}
\label{sec:non-abelian}

In this section we introduce the non-abelian cohomology $\Ho^1(\Gamma,G)$ of a cyclic group $\Gamma$ with values in a simple group $G$. We show in \cref{thm:H1tauG} that this space can be realized as a quotient  of the non-abelian cohomology of $\Gamma$ with values in a $\Gamma$-invariant torus of $G$. This will be a key result used in \cref{sec:local} and in \cref{sec:localtypes}. We postpone to \cref{appendix} further consequences of \cref{thm:H1tauG}, including an explicit classification of finite order automorphisms of simple groups (see \cref{thm:clas}).

\medskip

We begin by recalling the definition of $\Ho^1(\Gamma,A)$ for a group $A$, not necessarily abelian. Fix a generator $\gamma$ of $\Gamma$ and let $m$ be the cardinality of $\Gamma$. By abuse of notation, we will still denote by $\gamma$ the automorphism of $G$ induced by $\gamma$.  Throughout we will assume that $\car(k)$ does not divide the order of $\Gamma$.

\begin{defi} We say that an element $g \in A$ is a \textit{cocycle} if  
\begin{equation} \label{eq:cocycle} g \cdot \gamma(g) \cdots  \gamma^{m-1}(g)=1
\end{equation} holds. The set of cocycles is denoted by $Z^1(\Gamma,A)$. The \textit{non-abelian cohomology of $\Gamma$ with values in $A$}, denoted by $\Ho^1(\Gamma,A)$, is the quotient of $Z^1(\Gamma,A)$  by the equivalence relation $\sim_\gamma$, which identifies two cocycles $a$ and $b$ if and only if there exists an element $g \in A$ such that $a = g b \gamma(g)^{-1}$.\end{defi}

\begin{rmk} This recovers the definition of non-abelian cohomology introduced in \cite[Part I, \S 5]{serre:1997:galois}, where a 1-cocycle is defined as being a map $\tau \to a_\tau$ of $\Gamma$ to $A$ such that 
\begin{equation} \label{eq:serrecocylce} a_{\tau_1 \tau_2} = a_{\tau_1} \tau_1(a_{\tau_2})
\end{equation}
for all $\tau_1, \tau_2 \in \Gamma$. Two 1-cocycles $a$ and $b$ are \textit{cohomologous} if there exists $c \in A$ such that $a_{\tau} = c b_{\tau} \tau(c)^{-1}$ for every $\tau \in \Gamma$. Since $\Gamma$ is cyclic, a 1-cocycle is uniquely determined by the assignment $\gamma \to a_\gamma$, for a generator $\gamma$ of $\Gamma$ and the condition \eqref{eq:serrecocylce} is translated into \eqref{eq:cocycle}, recovering the definition of cocycles presented here.
Similarly, the equivalence relation $\sim_\gamma$ translates the property of two 1-cocycles being cohomologous.\end{rmk}

One can see that $\Ho^1(\Gamma,A)$ is a pointed space, with the class of $1 \in A$ being a preferred point, but it is not a group if $A$ is not abelian. Furthermore, given an exact sequence of groups
\[ \xymatrix{ 0 \ar[r] & A \ar[r] & B \ar[r] & C \ar[r] & 0 }\] which are equipped with compatible actions of $\Gamma$, one obtains an exact sequence
\[\xymatrix{ 0 \ar[r] & A^\Gamma \ar[r] & B^\Gamma \ar[r] & C^\Gamma \ar[r] & \Ho^1(\Gamma,A) \ar[r] & \Ho^1(\Gamma,B) \ar[r] & \Ho^1(\Gamma,C) }.
\] In this way $\Ho^1(\Gamma,A)$ can be interpreted as a space which measures the failure of right-exactness for the functor that takes a group with a $\Gamma$ action to its subgroup of $\Gamma$-invariants.

\subsection{Non-abelian cohomology for diagram automorphisms} \label{sec:cohomology_diagram}

Let $G$ be a simple algebraic group over $k$ and assume that the cyclic group $\Gamma$ acts on $G$, hence on $G(k)$. Throughout this section, we will further assume that the generator $\gamma$ of $\Gamma$ acts on $G$ via a diagram automorphism $\tau$ of order $r$ preserving a maximal torus $T$ and a Borel subgroup $B$ containing $T$. We will denote  $\Ho^1(\Gamma,G(k))$ by $\Ho^1_\tau(\Gamma,G)$.

Since $\tau$ is a diagram automorphism, the Weyl group $W$ of $(G,T)$ is acted on by $\Gamma$ and we denote by $W^\tau=W^\Gamma$ the invariant elements. Note that if $R=R(G,T)$ is the root system of $G$ relative to $T$, then $W^\tau$ is the Weyl group of $G^\tau$. 

\begin{prop} \label{prop:H1GTW} Using the above notation, there is a bijection 
\[\Ho^1_\tau(\Gamma,G) \cong \Ho^1_\tau(\Gamma,T) / W^\tau.
\] 
\end{prop}

We learned from \cite{Pappas-Rapoport:2022}, that one could also deduce this result from the proof of  \cite[Proposition 2.4]{pappas.zhu:2013}. We give here a different proof, relying on \cite{mohrdieck:2003:conjugacy} which focuses on the invariant theory of twisted conjugacy classes of simple algebraic groups. 

\smallskip

We first show that $\Ho^1_\tau(\Gamma,G)$ is a quotient of $\Ho^1_\tau(\Gamma, T)$. This fact actually holds without the assumption that $\gamma$ acts on $G$ by a diagram automorphism, and so we state and prove this result without this assumption. 

\begin{lemma}\label{lem:TsurgG} Assume that the generator $\gamma$ of $\Gamma$ acts on the simple group $G$ by an automorphism $\sigma$. Let $T$ be a maximal torus preserved by $\sigma$ and B a Borel subgroup containing T which is also preserved by $\sigma$. The natural inclusion $B \subseteq G$ and projection $B \to T$ induce a surjection $\Ho^1_\sigma(\Gamma,T) \to  \Ho^1_\sigma(\Gamma,G)$. \end{lemma}

\begin{proof} We first of all recall \cite[Lemma 7.3]{steinberg:1968:endomorphisms}, which  ensures that the map $G \times B \to G$ defined by $(g,b) \mapsto g b \sigma(g)^{-1}$ is surjective. This implies that we can write every element $x \in Z^1(\Gamma,G) \subseteq G$ as $x= g b \sigma(g)^{-1}$ for some $g \in G$ and $b \in B$. Note that since $x \sigma(x) \cdots \sigma^{r-1}(x)=1$ necessarily we have that also $b \sigma(b) \cdots \tau^{r-1}(b)=1$, that is $b \in Z^1(\Gamma,B) \subset Z^1(\Gamma,G)$. The equality $x= g b \sigma(g)^{-1}$ tells us that $x$ and $b$ define the same element of $\Ho^1_\sigma(\Gamma,G)$ and so that the inclusion of $B$ in $G$ induces a surjection  $\Ho^1_\sigma(\Gamma,B) \to  \Ho^1_\sigma(\Gamma,G)$. 

\smallskip 
We now show that $\Ho^1_\sigma(\Gamma,B) \cong \Ho^1_\sigma(\Gamma,T)$. To show this consider the exact sequence of algebraic groups
\begin{equation} \label{eq:UBT}\xymatrix{ 1 \ar[r] & U \ar[r] & B \ar[r] & T \ar[r] & 1}
\end{equation}
on which $\Gamma$ acts.

Denote by $U^{(\ell)}=[U^{(\ell-1)},U^{(\ell-1)}]$ the derived series of $U$, and set $U_\ell:=U^{(\ell)}/U^{(\ell+1)}$. By convention, $U^{(0)}=U$.  Then $U_\ell$ can be regarded as a vector space over $k$. Let $N$ be such that $U^{(N)} \neq 0$ and $U^{(N+1)} =0$.
For any linear action of $\Gamma$ on $U_\ell$, we have $H^i(\Gamma, U_\ell)=0$ for $i\geq 1$, since the characteristic of $k$ does not divide $|\Gamma|$. Note that the group $U^{(\ell)}$ is a normal subgroup of $B$, and setting $B_\ell:=B/U^{(\ell)}$ we obtain that $U_\ell$ is normal in $B_{\ell +1}$ and that $B_\ell = B_{\ell+1}/U_\ell$. Then, by \cite[Corollary 2, \S 5.5]{serre:1997:galois} and \cite[Corollary, \S 5.6]{serre:1997:galois}, there exists a natural bijection
\[ H^1_\sigma(\Gamma, B_\ell)\cong H_\sigma^1(\Gamma, B_{\ell+1})  \]  for any $\ell$.
Note that $B_0=T$ and $B_{N+1}=B$. It then follows that $\Ho^1_\sigma(\Gamma,B) \cong \Ho^1_\sigma(\Gamma,T)$.
\end{proof}

In order to show that \cref{prop:H1GTW} holds, we will use \cite[Theorem 1.1]{mohrdieck:2003:conjugacy}. This result states that the inclusion map $T \subset G$ induces an isomorphism of schemes 
\[\xymatrix{ T_\tau/ W^\tau \ar[r]^\cong & G/\!/_\tau G,}
\] where $T_\tau$ consists of the coinvariant $T/(1-\tau)T$ and $G/\!/_\tau G$ is the categorical quotient of $G$ by the action of $G$ on itself defined as $g*h := g h \tau(g)^{-1}$. 

Combining this statement with \cref{def:parahoric}, we obtain the commutative diagram
\[\xymatrix{\Ho^1_\tau(\Gamma,T) \ar@{->>}[r]\ar[d] & \Ho^1_\tau(\Gamma,G) \ar[d] \ar@{^(->}[r] & G(k)/\sim_\tau \ar[dl] \\
T_\tau / W^\tau \ar[r]^\cong & (G/\!/_\tau G)(k)
}\]
where $\sim_\tau$ identifies two elements $g$ and $g'$ of $G(k)$ if and only if $g = h g' \tau(h)^{-1}$ for some $h \in G(k)$. To conclude the proof of \cref{prop:H1GTW}, we are then left to show that the map $\Ho^1_\tau(\Gamma,G) \to (G/\!/_\tau G)(k)$ is injective. 

\smallskip 

Since $\Gamma$ acts on $G$, we can build the disconnected group scheme $\widetilde{G}=G \rtimes \Gamma$, whose elements will be denoted $g.\gamma$ for $g \in G$ and $\gamma_i \in \Gamma$. Using this language $G/\!/_\tau G$ is the categorical quotient of the component $G.\gamma$ by conjugation and $G(k)/\sim_\tau=G(k).\gamma/\!\sim$ consists of conjugacy classes of elements in the component $G.\gamma$. In \cite[Proposition 3.18]{mohrdieck:2003:conjugacy}, it is shown that every fiber of the map $G.\gamma \to G/\!/_\tau G$ contains exactly one orbit consisting of semisimple elements. In particular the restriction of $G(k).\gamma/\!\sim  \; \to (G/\!/_\tau G)(k)$ to semisimple elements in $G.\gamma$ is injective. We then focus our attention on semisimple elements of $\widetilde{G}$.

As for connected groups, an element $g.\gamma$ of $\widetilde{G}$ is semisimple if there exists a linear representations $\rho \colon \widetilde{G} \to \text{GL}_n$, such that $\rho(g.\gamma)$ is diagonalizable. For a connected group $G$, an element is semisimple if and only if it is conjugated to an element of the torus $T$. In the following lemma, we will see that an analogous statement holds in this situation. Using the terminology introduced in \cite{springer:2006:twisted}, we say that an element $g \in G$ is \textit{$\tau$-semisimple} if and only if $g$ is $\tau$-conjugated to an element of $T$, i.e. there exists $h \in G$ such that $h g \tau(h)^{-1} \in T$.

\begin{lemma} An element $g \in G$ is $\tau$-semisimple if and only if $g.\gamma$ is semisimple in $\widetilde{G}$.
\end{lemma}

\begin{proof} We first of all prove that if $g.\gamma$ is semisimple, then $g$ is $\tau$-semisimple by applying \cite[Proposition 3.4]{mohrdieck:2003:conjugacy}. Applying that result to the element $z=1.\gamma \in \widetilde{G}$ we indeed obtain that $g.\gamma$ is conjugate, via an element $h.1 \in \widetilde{G}$ to an element of $T^{\tau,0}.\gamma$. Since $\gamma$ acts on $G$ by $\tau$, this is equivalent to $g$ being $\tau$-conjugate to an element of $T^{\tau,0}$, hence being $\tau$-semisimple.

Conversely, assume that $g=ht\tau(h)^{-1}$ for some $h \in G$ and $t \in T$. This is equivalent to the element $g.\gamma \in \widetilde{G}$ being  conjugated to $t.\gamma$ via the element $h.1$. Since semisimplicity of an element in invariant under conjugation, it is then enough to show that $t.\gamma$ is a semisimple element in $\widetilde{G}$. This is a direct consequence of the last part of \cite[Lemma 3.5]{mohrdieck:2003:conjugacy} using $z=1.\gamma$.
\end{proof}

\begin{proof}[Proof of \cref{prop:H1GTW}] By the above arguments, it is enough to show that every element of $\Ho^1_\tau(\Gamma,G)$ can be represented by a $\tau$-semisimple element of $G$. This is equivalent to showing that $\Ho^1_\tau(\Gamma,T)$ surjects onto $ \Ho^1_\tau(\Gamma,G)$, and this is guaranteed by \cref{lem:TsurgG}.
\end{proof}
\subsection{Cohomology of tori} \label{sec:cohomtori} Given the results of \cref{prop:H1GTW}, in order to give an explicit description of the local types of $(\Gamma, G)$-bundles through non abelian cohomology, we focus our attention on the computation of $\Ho^1(\Gamma,T)$.

\begin{lemma} \label{lem:eq*} Let $T_m$ denote the $m$-th torsion elements of $T$. 
  The inclusion map $T_m \to T$ induces an isomorphism $\Ho^1_\tau(\Gamma,T) \cong \Ho^1_\tau(\Gamma,T_m)$.    
\end{lemma}

\begin{proof} Let $(-)^m$ denote the map $T \to T$ given by $t \mapsto t^m$. From the exact sequence
\[ \xymatrix{ 0 \ar[r]& T_m \ar[r] & T \ar[r]^{(-)^m} & T \ar[r] & 0
}\] we obtain the long exact sequence
\[\xymatrix{ T^\tau \ar[r]^-{{(-)^m}} & T^\tau \ar[r] & \Ho^1_\tau(\Gamma,T_m)\ar[r] &\Ho^1_\tau(\Gamma,T)\ar[r]^-{{(-)^m}} &\Ho^1_\tau(\Gamma,T).
}\] Observe that if $t$ is a cocycle, i.e. belongs to $\text{Z}^1(\Gamma,T)$, then $t^m$ is zero in $\Ho^1_\tau(\Gamma,T)$ since it can be written as $s\tau(s)^{-1}$ for $s=\prod_{i=0}^{m-2}\tau^i(t)^{m-1-i}$. It follows that $\Ho^1_\tau(\Gamma,T_m)$ surjects onto $\Ho^1_\tau(\Gamma,T)$. Moreover, since $G$ is simply connected (or adjoint) the group $T^\tau$ is connected, so that the induced map $(-)^m$ between invariants is necessarily surjective, hence $\Ho^1_\tau(\Gamma,T_m)$ injects into $\Ho^1_\tau(\Gamma,T)$. Thus we obtain $\Ho^1_\tau(\Gamma,T_m)\cong \Ho^1_\tau(\Gamma,T)$ as claimed. \end{proof}

\begin{prop} \label{prop:H1Tcomb} Let $\Gamma$ be a cyclic group of order $m$ acting on the simple algebraic group $G$ over $k$ by the (diagram) automorphism $\tau$ of order $r$. If  $\car(k)$ does not divide $m$ and $G$ is either simply connected or adjoint, then we have a natural identification
\[ \Ho^1_\tau(\Gamma,T) \cong \frac{\tfrac{r}{m}\coX(T)^\tau}{\norm_\tau(\coX(T))}=\frac{\tfrac{1}{m}\coX(T)^\tau}{\Av_\tau(\coX(T))},\]
where $\norm_\tau$
 is the norm operator 
 $1+ \tau + \dots + \tau^{r-1}$
  and $\Av_\tau=\frac{1}{r}\norm_\tau$.
\end{prop}

\begin{proof} We first of all observe that since the group $\Gamma$ is cyclic, one recovers that $\Ho^2_\tau(\Gamma,\coX(T))$ is isomorphic to $\coX(T)^\tau/\norm_{\tau,m}(\coX(T))$ (see, e.g. \cite[Example 2, page 58-59]{brown:1982:cohomology}), where $\norm_{\tau,m}$ denotes $1+\tau + \dots + \tau^{m-1}$. It thus remains to prove that $\Ho^1_\tau(\Gamma,T) \cong \Ho^2 (\Gamma,\coX(T))$. Moreover, in view of \cref{lem:eq*}, we have
\begin{equation} \label{eq:*}\Ho^1_\tau(\Gamma,T) \cong \Ho^1_\tau(\Gamma,T_m), \end{equation} where $T_m$ denotes the $m$-th torsion elements of $T$.

We next have the exact sequence of abelian groups
\begin{equation}\label{eq:XmXT} \xymatrix{0 \ar[r] &\coX(T) \ar[r]^{m\cdot} & \coX(T)\ar[r] & T_m \ar[r] &0,
}
\end{equation} where we have identify $T_m$ with the quotient $\coX(T)/m\coX(T)$ as follows. Choose a basis $\lambda_1,\dots, \lambda_\ell$ of $\coX(T)$, so that $\coX(T) \cong \ZZ^\ell$. Similarly, we fix an isomorphism $T \cong (k^\times)^\ell$, which identifies $T_m$ with $\mu_m^\ell$. The map $\ZZ^\ell \to (k^\times)^\ell$ given by $(a_1,\dots,a_\ell) \mapsto (\zeta^{a_1},\dots, \zeta^{a_\ell})$, induces the wanted isomorphism. We can then take $\tau$-invariants of \eqref{eq:XmXT}, obtaining the short exact sequence
\[\xymatrix{0 \ar[r] &\Ho^1_\tau(\Gamma,\coX(T)) \ar[r] & \Ho^1_\tau(\Gamma,T_m) \ar[r] & \Ho^2_\tau(\Gamma,\coX(T))\ar[r] &0,
}
\] where we have used the fact that the multiplication map $m\, \cdot \; \colon \coX(T) \to \coX(T)$ induces the zero maps $\Ho^1_\tau(\Gamma,\coX(T)) \to \Ho^1_\tau(\Gamma,\coX(T))$ and $\Ho^2_\tau(\Gamma,\coX(T)) \to \Ho^2_\tau(\Gamma,\coX(T))$ (see e.g., \cite[\S 7.7]{sharifi:notes}).

In order to conclude, it suffices to show that $\Ho^1_\tau(\Gamma,\coX(T))=0$. When $G$ is simply connected (or adjoint), we can identify $\coX(T)$ with $\ZZ^\ell$, with the action of $\tau$ given  by either
\[\tau(a_1, \dots, a_\ell) = (a_\ell, \dots, a_1) \qquad \text{if $\tau$ has order $2$}\]
or
\[
\tau(a_1, a_2, a_3)= (a_2,a_3,a_1) \qquad \text{if $\tau$ has order $3$ (and so necessarily $G$ of type $D_4$)}.
\] In both cases, a direct computation shows that $\Ho^1_\tau(\Gamma,\ZZ^\ell)=0$, implying the isomorphism $\Ho^1_\tau(\Gamma,T_m) \cong \Ho^2_\tau(\Gamma, \coX(T))$ which, combined with \eqref{eq:*}, allows us to conclude.\end{proof}

\begin{rmk} \label{rmk:overC} When $k=\CC$, one does not need to assume that $G$ is adjoint or simply connected to obtain the isomorphism of \cref{prop:H1Tcomb}. In fact we have the exact sequence
\begin{equation}  \label{eq:XhT} \xymatrix{ 0 \ar[r] & \coX(T) \ar[r] &  \mathfrak{h} \ar[r]^{e} & T \ar[r] & 0, }
\end{equation} where $e(h):=\exp(2\pi i h)$ and we can identify $\mathfrak{h}$ with $\coX(T) \otimes_\ZZ \CC$. 

The action of $\Gamma$ on $T$ induces an action on $\mathfrak{h}$ and on $\coX(T)$, which we still denote $\tau$. Taking $\tau$-invariants of \eqref{eq:XhT} (seen as an exact sequence of abelian groups) we obtain the following long exact sequence of abelian groups
\begin{equation*}  \label{eq:XhTlong} \xymatrix@C=1.5em{ 0 \ar[r] & \coX(T)^\tau \ar[r] &  \mathfrak{h}^\tau \ar[r]^{e} & T^\tau \ar[r] & \Ho^1_\tau(\Gamma, \coX(T))  \ar[r]& 0 \ar[r] & \Ho^1_\tau(\Gamma,T) \ar[r] & \Ho^2(\Gamma, \coX(T)) \ar[r] & 0,  }
\end{equation*}  hence the desired isomorphism claimed in \cref{prop:H1Tcomb}.

Moreover, given an element $[t] \in \Ho^1_\tau(\Gamma,T)$, we explicitly describe the associated equivalence class in $\frac{1}{m}\coX(T)^\tau/\Av_\tau(\coX(T))$. Choose $h \in \mathfrak{h}$ such that $t=e(h)$, then the natural map $\Ho^1_\tau(\Gamma,T) \cong \frac{1}{m}\coX(T)^\tau/\Av_\tau(\coX(T))$ identifies $[e(h)]=[t]$  to the element $\Av_\tau(h)= \frac{1}{m}[h+\tau(h)+ \dots + \tau^{m-1}(h)]$ of $\frac{1}{m}\coX(T)^\tau/\Av_\tau(\coX(T))$. This can be seen from the explicit computation of the cohomology of cyclic groups as in \cite[Example 2, page 58-59]{brown:1982:cohomology}. \end{rmk}

\subsection{Cohomology of simple groups}\label{sec:H1G} We now combine \cref{prop:H1GTW}, \cref{prop:H1Tcomb} and \cref{rmk:overC} to deduce the following result.

\begin{theorem} \label{thm:H1tauG} Let $\Gamma =\langle \gamma \rangle$ be the cyclic group of order $m$ such that $\gamma$ acts on $G$ via a diagram automorphism $\tau$ of order $r$. Then if either $k = \CC$ or $G$ is simply connected or adjoint (and $\car(k)$ does not divide $m$), then 
\[ \Ho^1_\tau(\Gamma,G) \cong \dfrac{\frac{1}{m}\coX(T)^\tau}{\Av_\tau(\coX(T)) \rtimes W^\tau} = \dfrac{\frac{r}{m}\coX(T)^\tau}{\norm_\tau(\coX(T)) \rtimes W^\tau},
\] where $\norm_\tau := r \Av_\tau$.
\end{theorem}

\section{Parahoric group schemes}\label{sec:local}

The main result of this section is to realize every parahoric group scheme through invariant sections of restriction of scalars. 

\subsection{Parahoric group schemes from invariants} Throughout this section  $G$ is a simple group over $k$  which is assumed to be simply connected. We will denote by $G_\ad$ the adjoint quotient of $G$.  Let $\Gamma$ be the cyclic group of order $m$ and assume that $\car(k)$ does not divide $m$. Assume that a generator of $\Gamma$ acts on $G$ (and on $G_\ad$) by a diagram automorphism $\tau$ preserving the maximal tori $T$ and $T_\ad$. In view of \cref{thm:H1tauG}, for any $\kappa\in \Ho^1_\tau(\Gamma, G_\ad)$ we can choose $\kappa=[\bar{t}]$ for some $\bar{t}\in T_\ad^{\tau}$. We then let $\sigma$ be the automorphism $\Ad_{\bar{t}} \circ \tau$. Then $\sigma^m=1$. 
Moreover, $\bar{t}$ can be written as $\zeta_m^{\lambda}$, for $\lambda \in \coX(T_\ad)$.
The group $\Gamma$ acts also on $\Kc_m$ by $\gamma(z^{\frac{1}{m}})=\zeta_m^{-1}z^{\frac{1}{m}}$. It follows that we can define two actions of $\Gamma$ on the group scheme $G_{\Kc_m}$, one induced by $\sigma$ and another one by $\tau$. We can then define two group schemes over $\Kc$, namely \[ G_{\sigma,\Kc}:=\Res_{\Kc_m/\Kc}(G_{\Kc_m})^\sigma\qquad \text{ and }\qquad G_{\tau,\Kc}:=\Res_{\Kc_r/\Kc}(G_{\Kc_r})^\tau,\]
whose global sections are contained in $G(\Kc_m)$.

We now show that these group schemes are naturally isomorphic. This can be interpreted as a group version of the Lie algebra isomorphism between $\g(\Kc_r)^\tau$ and $\g(\Kc_m)^\sigma$ (see e.g. \cite[Chapter 8]{kac:1990:infinite}). The key ingredient is to explicitly relate the two actions $\sigma$ and $\tau$ on $G_{\Kc_m}$. Choose an isomorphism between $\coX(T_\ad)$ and $\ZZ^\ell$ so that the element $\lambda$ corresponds to $(\lambda_1, \dots, \lambda_\ell)$. Similarly, we identify the torus $T_\ad$ with $\Gm^\ell$ and let $z^\theta$ denote the element of $T_\ad(\Kc_m)$ given by $(z^{\frac{\lambda_1}{m}}, \dots, z^{\frac{\lambda_\ell }{m}})$. In fact, $\theta=\frac{\lambda}{m}\in \frac{1}{m}\coX(T_\ad)^\tau$.

\begin{prop} \label{prop:kacgroup} The map $\Ad_{z^\theta}$ induces an isomorphism of group schemes
\[ \xymatrix{\Res_{\Kc_r/\Kc}(G_{\Kc_r})^\tau \, \ar[r]^-{\cong} & \Res_{\Kc_m/\Kc}(G_{\Kc_m})^\sigma}.\]
\end{prop}

\begin{proof} 
Let $\Gamma$ be the Galois group ${\rm Gal}(\Kc_m/\Kc)$ with a generator $\gamma$.  After the base change, we have the following isomorphism
\[ \Res_{\Kc_r/\Kc}(G_{\Kc_r})\times_{\Kc} \Kc_m\cong G_{\Kc_m}, \]
where the Galois group $ \Gamma$ acting on $\Res_{\Kc_r/\Kc}(G_{\Kc_r})\times_{\Kc} \Kc_m$ corresponds to the action of $\gamma$ on $G_{\Kc_m}$ given by combining $\tau$ on $G$ and the field automorphism $z\mapsto \zeta_m^{-1}z$. 
Similarly, we have 
\[ \Res_{\Kc_m/\Kc}(G_{\Kc_m})\times_{\Kc} \Kc_m\cong G_{\Kc_m}, \]
where the Galois group action corresponds to $\gamma$ acts on $ G_{\Kc_m}$ given by combining $\sigma$ on $G$, and the field automorphism as above. 
One may check easily that the automorphism $\Ad_{z^\theta}: G_{\Kc_m}\cong G_{\Kc_m}$ intertwines with the two actions given above. By Galois descent, we obtain the isomorphism 
$ \xymatrix{\Res_{\Kc_r/\Kc}(G_{\Kc_r})^\tau \, \ar[r]^-{\cong} & \Res_{\Kc_m/\Kc}(G_{\Kc_m})^\sigma}$. 
\end{proof}

\begin{theorem} 
\label{thm:local}
Let $\sigma$ be an automorphism of $G$ arising from a class $\kappa\in \Ho^1_\tau(\Gamma, G_\ad)$ as above. 
Under the isomorphism of  \cref{prop:kacgroup}, the group scheme $\Res_{\Oc_m/\Oc}(G_{\Oc_m})^\sigma$ is isomorphic to the parahoric group scheme $\Gc_{\theta}$, where $\theta=\frac{\lambda}{m}$.
\end{theorem}

\begin{proof} 
Let $B$ be the Borel subgroup preserved by $\tau$ which contains $T$. Let $B^-$ be the opposite Borel subgroup of $B$. Let $U^\pm$ be the unipotent radical of $B^\pm$. It is well-known that the multiplication $U^-\times U\times T\to G $ is an open embedding of varieties over $k$. It induces an open embedding of their jet schemes as pro-varieties over $k$:
\[ U^-(\mathcal{O}_m)\times U(\mathcal{O}_m)\times T(\mathcal{O}_m) \to G(\mathcal{O}_m).    \]
By taking $\sigma$-invariants, we get an open embedding:
\[    U^-(\mathcal{O}_m)^\sigma \times U(\mathcal{O}_m)^\sigma\times T(\mathcal{O}_m)^\sigma \to G(\mathcal{O}_m)^\sigma. \]
Since $G$ is simply-connected, $G^\sigma$ is a connected simple algebraic group, cf. \cite{steinberg:1968:endomorphisms}. Thus, $G(\mathcal{O}_m)^\sigma$ is also a connected pro-algebraic group over $k$. As a consequence, $U^-(\mathcal{O}_m)^\sigma, U(\mathcal{O}_m)^\sigma$ and $T(\mathcal{O}_m)^\sigma  $ generate $G(\mathcal{O}_m)^\sigma $. 

Via the isomorphism $\Ad_{z^\theta}: G_{\tau,\mathcal{K}} \cong  G_{\sigma,\mathcal{K} }$ given in \cref{prop:kacgroup}, we can identify the relative root systems for $G_{\tau, \mathcal{K}}$ and $G_{\sigma, \mathcal{K}}$. For any relative root $a\in \mathcal{R}$, $\mathcal{U}_a$ denotes the root subgroup in $G_{\tau, \mathcal{K}}$ and we denote $\mathcal{U}'_a$ the associated root subgroup in $G_{\sigma, \mathcal{K}}$. Moreover, there is a natural valuation structure on $\mathcal{U}'_a$. Using the description of the root subgroups given in \cref{sec:parahoric} and the analogue for $\Uc_{a}'$, by direct calculation we have for any $a\in \mathcal{R}$ the identification 
\[ \Ad_{z^\theta} ( \mathcal{U}_{a, -\theta(a)} )= \mathcal{U}'_{a, 0}.    \]
Moreover,  $\Ad_{z^\theta}(T(\mathcal{O}_r)^\tau )= T(\mathcal{O}_m)^\sigma$. Let $\mathcal{P}^\pm_\theta$ be the subgroup of $\mathcal{P}$ generated by $\{\mathcal{U}_{a, -\theta(a)}  \,: \,  a\in \mathcal{R}^\pm \}$. Then, $\Ad_{z^\theta}( \mathcal{P}^\pm_\theta)=U^\pm(\mathcal{O}_m)^\sigma$. 

It follows that $\Ad_{z^\theta}(  \mathcal{P}_\theta)= G(\mathcal{O}_m)^\sigma$. In view of \cite[I.7.6]{bruhat.tits:1984:II}, the isomorphism $\Ad_{z^\theta} \colon G_{\tau,\mathcal{K}}\cong G_{\sigma, \mathcal{K}}$ uniquely extends to an isomorphism $\mathcal{G}_\theta\cong {\rm Res}_{\mathcal{O}_m/\mathcal{O} }(G_{\mathcal{O}_m })^\sigma$. 
\end{proof}

\begin{rmk} \label{rmk:adjoint} In \cref{thm:local}, we have assumed that $G$ is simply connected, so that we can ensure that $G^\sigma$ and $G(\Oc_m)^\sigma$ are connected. When $G$ is not simply connected, this property doesn't automatically hold, but it will depend on the type of the group $G$. When $G$ is adjoint and $\tau$ has order $2$ (resp. $3$), this will be the case when $G$ is of type $A_{2\ell}$ or $E_6$ (resp. $D_4$), cf.\,\cite[9.8]{steinberg:1968:endomorphisms}. So in these cases, \cref{thm:local} also holds when $G$ is of adjoint type.  \end{rmk}

In the above, we consider the automorphism $\sigma$ arising from a class $\kappa\in \Ho^1_\tau(\Gamma, G_\ad)$. In the following proposition and its corollary, we obtain a decomposition of every finite order automorphism of $G$. We allow here $G$ to be a simple group which is not necessarily simply connected.

\begin{prop} \label{cor:sigma=tauAdtbar} Let $G$ be a simple algebraic group over $k$. Let $\sigma$ be an automorphism of $G$ of order $m$ and assume that $\car(k)$ does not divide $m$. Then we have a decomposition \[\sigma = \Ad_{\bar{t}} \circ \tau,\] where $\tau$ is a diagram automorphism of $G$ preserving a Borel $B$ and a maximal torus $T$ contained in $B$, and $\bar{t}$ is a $\tau$-invariant element of $T_\ad$ which can be written as $\zeta^\lambda_m$ for $\lambda \in \coX(T_\ad)^\tau$.
\end{prop}

\begin{proof} We first assume $G$ is adjoint. Then, we write $\sigma=\Ad_g \circ \tau'$, for some diagram automorphism $\tau'$ preserving a Borel subgroup $B'$ and a maximal torus $T'$ contained in $B'$.  By the condition $\sigma^m=1$ and the fact that $G$ is adjoint, we deduce that $g$ satisfies the condition $g\tau'(g)\cdots (\tau')^{m-1}(g)=1$. 
Thus, $[g]\in \Ho^1_{\tau'}(\Gamma, G)$ where $\Gamma$ is the cyclic group of order $m$. By \cref{thm:H1tauG}, there exists $t'\in T'$ such that $[g]=[t']$, with $t'=(\zeta_m)^\lambda$ for some $\lambda\in \coX(T')^{\tau'}$. One may observe that $[g]=[t']$ is equivalent to that, there exists an element $k\in G$, such that $\sigma=\Ad_k \circ \Ad_{t'}\circ \tau'\circ  \Ad_{k^{-1}}$. Thus, $\sigma=\Ad_t\circ \tau $, where $\tau=\Ad_k\circ \tau'\circ \Ad_{k^{-1}}$ is a diagram automorphism preserving $B=\Ad_{k}(B')$ and $T=\Ad_t(T')$, and $t=\Ad_k(t')=\zeta_m^{\lambda}\in T$ with $\lambda$ regarded as an element in $\coX(T)^\tau$. 

Now, let $G$ be a simple algebraic group which is not adjoint. Passing to $\bar{G}=G_{\ad}=G/Z(G)$, we can write $\sigma=\Ad_{\bar{t}}\circ \tau $ where $\tau$ preserves a Borel subgroup $\bar{B}$ of $G_{\ad}$ and a maximal torus $\bar{T}$ contained in $\bar{B}$, and $\bar{t}\in \bar{T}^\tau$. Denoting by $T$ and $B$ the preimages of $\bar{T}$ and $\bar{B}$ in $G$, the result is proved.  \end{proof}

\begin{cor}\label{prop:decom}
Let $G$ be a simple algebraic group over $k$. Let $\sigma$ be an automorphism on $G$ of finite order. Assume that $\car(k)$ does not divide the order of $\sigma$ and the order of the center of $G^{\eta,0}$, where $\eta$ is a diagram automorphism such that $\bar{\sigma}=\bar{\eta} \in \Out(G)$. Then, we have a decomposition 
\[\sigma=\Ad_{t} \circ \tau,\]
where $\tau$ is a diagram automorphism preserving a Borel subgroup $B$ and a maximal torus $T$ contained in $B$, and $t=\zeta_{M}^{\lambda}$ for some positive integer $M$ and $\lambda\in X_*(T)^\tau$. In fact, one can take $M=|\sigma|\cdot |Z(G^{\tau,0})|$, where $Z(G^{\tau,0})$ is the center of $G^{\tau,0}$. 
\end{cor}

\begin{proof} It follows from \cref{cor:sigma=tauAdtbar} that we can write $\sigma=\Ad_{\bar{t}}\circ \tau $ where $\tau$ preserves a Borel subgroup $\bar{B}$ of $G_{\ad}$ and a maximal torus $\bar{T}$ contained in $\bar{B}$, and $\bar{t}\in \bar{T}^\tau$. Denote by $T$ and $B$ the preimages of $\bar{T}$ and $\bar{B}$ in $G$. Let $\bar{T}_m$ be the subgroup of $\bar{T}$ consisting of those elements $\bar{t}$ such that $\bar{t}^m=1$.
The surjective map $T \mapsto \bar{T}$ induces the surjective morphism $\pi:T^{\tau,0} \to \bar{T}^\tau$, 
which restricts to the surjection $T^{\tau,0}_{M} \to \bar{T}^\tau_m$ where $M= m\cdot |Z(G^{\tau,0 })|$, as $\ker(\pi)=Z(G^{\tau,0})$ (see, e.g.,\cite[Corollary 7.6.4 (iii)]{Springer:2009}).  In particular this means that for such an $M$, there exists $t\in T_{M}^{\tau,0}$ such that $\bar{t}$ is the image of $t$ in $\bar{T}$. Moreover, by assumption $\car(k)$ does not divide $M$, hence there exists an $M$-th primitive root of unity $\zeta_{M}$. Observe that the assignment  $\mu \mapsto \zeta_{M}^\mu$ induces a surjective map $\coX(T)^\tau = \coX(T^{\tau,0}) \to T^{\tau,0}_{M}$.  It follows that we can write $t=\zeta_{M}^\lambda$ for some $\lambda\in \coX(T)^\tau$.
Thus, $\sigma=\Ad_t\circ \tau$ on $G$ with $t=\zeta_{M}^\lambda$ for some $\lambda\in \coX(T)^\tau$.
\end{proof}

We can further obtain the following generalization of \cref{thm:local}.

\begin{cor}\label{cor:parahoric}
Let $G$ be a simply-connected simple algebraic group with a finite order automorphism $\sigma$ such that $\sigma^m=1$, and $\car(k)$ does not divide $m$.  Then, ${\rm Res}_{\mathcal{O}_m/\mathcal{O}}(G_{\mathcal{O}_m})^\sigma$ is a parahoric group scheme. 
\end{cor}
\begin{proof}
By \cref{cor:sigma=tauAdtbar}, we are able to write $\sigma= \Ad_{\bar{t}} \circ \tau$ for some diagram automorphism $\tau$ preserving a maximal torus $T$, and an element $\bar{t}\in T_\ad$ such that $\bar{t}^m=1$. It then follows from 
\cref{thm:H1tauG} that $[\bar{t}] \in \Ho_\tau(\Gamma,G_\ad)$. We are then in the assumptions of  \cref{thm:local}, so that we can conclude that ${\rm Res}_{\mathcal{O}_m/\mathcal{O}}(G_{\mathcal{O}_m})^\sigma$  is a parahoric group scheme as desired.
\end{proof}

\subsection{Matching with Bruhat--Tits building of \texorpdfstring{$G(\mathcal{K}_r)^\tau$}{(G(K)r)-tau}}
We assume $G$ is simply-connected. Following \cite{Tits:1979}, we describe the chamber of the Bruhat--Tits building of the twisted loop group $G(\mathcal{K}_r)^\tau$ for the maximal split $\mathcal{K}$-torus $T^\tau\times_k \mathcal{K}$, and the corresponding affine Weyl group action. The chamber corresponding to $T^\tau\times_k \mathcal{K}$ is  given by $X_*(T)^\tau_{\mathbb{R}}$. The group $W^\tau$ can be regarded as the Weyl group of the relative root systems corresponding to $T^\tau\times_k \mathcal{K}$. The translation part comes from $T(\mathcal{K}_r)^\tau/ T(\mathcal{O}_r)^\tau$. Note that there is a natural identification  $X_*(T)_\tau\cong T(\mathcal{K}_r)^\tau/ T(\mathcal{O}_r)^\tau$ given by $\bar{\lambda}\mapsto \prod_{i=0}^{r-1} \tau^i(t^{\frac{\lambda}{r} })$ for any $\bar{\lambda}\in X_*(T)_\tau$,  where $\lambda$ is a representative of $\bar{\lambda}$ in $X_*(T)$ \cite[Section 2.3]{Besson-Hong:2020}.  Then the translation lattice on the chamber $X_*(T)^\tau_\mathbb{R}$ is exactly given by $\Av_r(X_*(T))$.  

Given an element $\theta\in X_*(T)^\tau_\mathbb{Q}$, let $m$ be the minimal positive integer such that $\theta\in \frac{1}{m} X_*(T_\ad)^\tau$. We write $\theta=\frac{\lambda}{m}$ for some $\lambda\in X_*(T_\ad)^\tau$ and set $\sigma= \Ad_{\zeta_m^\lambda} \circ \tau$ for some $m$-th primitive root $\zeta_m$ of unity. The following corollary  follows from \cref{thm:local}. 
\begin{theorem}\label{thm:BT}
If $\car(k)$ does not divide $m$, then there exists an isomorphism of group schemes 
\[ \Res_{\mathcal{O}_r/\mathcal{O} }( G_{\mathcal{O}_r } )^\sigma \cong \mathcal{G}_\theta,  \]
where $\mathcal{G}_\theta$ is the parahoric group scheme corresponding to $\theta$. 
\end{theorem}

From the Bruhat--Tits theory, parahoric group schemes are determined by facets, i.e. any interior points of a given facet give rise to the same parahoric group scheme. From this corollary, we can try to find an interior point in a given facet such that $m$ is minimal. Thus, the restriction of the characteristic $p$ can be very small. 

\begin{rmk}
Let $\mathcal{G}$ be a special parahoric group scheme over $\mathcal{O}$ of type $X_N^{(r)}$ with $r>1$. Here, ``special'' refers to that the parahoric group schemes correspond to special vertices of Bruhat--Tits building of $G(\mathcal{K}_r)^\tau$. When $X_N^{(r)}\not= A_{2\ell}^{(2)}$, the parahoric group scheme $\mathcal{G}$ is isomorphic to $ {\rm Res}_{\mathcal{O}_{r}/\mathcal{O}}( G_{\mathcal{O}_{r} }  )^\sigma $ where $\sigma$ is a diagram automorphism of order $r$. When $X_N^{(r)}=A_{2\ell}^{(2)}$, there are two special parahoric group schemes which are not isomorphic. One of them can be realized using a diagram automorphism of order $2$,  while another one can be realized by a standard automorphism of order $4$, see \cite{Besson-Hong:2020,Hong-Yu:2022}.
\end{rmk}

\section{\texorpdfstring{$(\Gamma,G)$}{(Gamma,G)}-bundles} \label{sec:localtypes}

\subsection{\texorpdfstring{$(\Gamma,G)$}{(Gamma-G)}-bundles and their local types} We recall in this section the notion of $(\Gamma,G)$-bundles, following \cite{balaji.seshadri:2015:moduli, damiolini:2021:equivariant}. Let $\Gamma$ be a finite group acting on a smooth and projective curve $C$. Assume further that $G$ is an affine group scheme over $k$ and fix an action of $\Gamma$ on $G$ via $\rho \colon \Gamma \to \text{Aut}(G)$. Throughout we will assume that $G$ is smooth and that the characteristic of $k$ does not divide $|\Gamma|$.

\begin{defi} A $(\Gamma,G)$-bundle is the data of a right $G$-bundle $\Ec$ over $C$ together with a left action of $\Gamma$ on its total space, lifting the action of $\Gamma$ on $C$, and compatible with the action of $\Gamma$ on $G$.\end{defi}

\begin{rmk} One can similarly define a $(\Gamma,G)$-bundle as being a left $G$-bundle with a right action of $\Gamma$ on its total space. It will be clear from the context which version we are going to use.
\end{rmk}

More explicitly, on every $\Gamma$-equivariant open set $U \subseteq C$ on which $\Ec$ is trivial as a $G$-bundle, we can express the compatibility condition as follows. Let $1 \in \Ec(U)$ define a trivialization of $\Ec$, then for all $\gamma \in \Gamma$ and $g \in G(U)$ we require that
\begin{equation}\label{eq:compatibility} \gamma(1 \cdot g) = \gamma(1) \cdot \rho_\gamma(g).
\end{equation} This in particular tells us that the action of $\Gamma$ on $\Ec$ is uniquely determined by the action of $\Gamma$ on the preferred section $1$ and the fixed action of $\Gamma$ on $G$ by $\rho$. 

We now consider the case in which $U$ is replaced by the disk $\Db_x=\Spec(k[\![z_x]\!])$ around a point $x \in C$. The stabilizer $\Gamma_x$ of $x \in C$ is a cyclic group, denote its generator by $\gamma_x$, and denote by $m_x$ its order. Then, on the formal disk $\Db_x$ about $x$ the element $\gamma_x$ acts multiplying the local coordinate $z_x$ by a primitive $m_x$-th root of unity. Any $(\Gamma,G)$-bundle $\Ec$ on $C$ trivializes, as a $G$-bundle, over $\Db_x$. We fix a trivialization induced by the element $1 \in \Ec(\Db_x)$, so that we can, and will, identify $\gamma_x(1)$ with a unique element of $G(\Db_x)$ and $1$ with the unit of $G(\Db_x)$. We then deduce from \eqref{eq:compatibility} and $\gamma_x^{m_x}=e$, that $\gamma_x(1)$
needs to satisfy the cocycle condition 
\[\gamma_x(1) \rho_\gamma(\gamma_x(1)) \cdots \rho_\gamma^{m_x-1}(\gamma_x(1)) =1.
\] 
Changing the trivialization to $1' \in \Ec(\Db_x)$, we obtain that 
\[\gamma_x(1') = a \cdot  \gamma_x(1) \cdot \rho_\gamma^{-1}(a). 
\] for $a \in G(\Db_x)$ such that $1=1'\cdot a$.  It follows that isomorphism classes of $(\Gamma_x,G)$-bundles on $\Db_x$ are described by the non abelian cohomology $\Ho^1(\Gamma_x,G(\Db_x))$, where we view $G(\Db_x)$ as a $\Gamma_x$-group through the combined actions of $\Gamma_x$ on $\Db_x$ and on $G$ (via $\rho$). 

Similarly to \cite[Lemma 2.5]{teleman.woodward:2003:parabolic} we have the following statement.

\begin{prop} \label{prop:twgen} The closed embedding $x \to \Db_x$ (or equivalently the evaluation map $k[\![z_x]\!] \to k$ given by $z_x \mapsto 0$), induces an isomorphism $\Ho^1(\Gamma_x,G(\Db_x)) \cong \Ho^1(\Gamma_x,G(k))$. In particular, isomorphism classes of $(\Gamma_x,G)$-bundles on $\Db_x$ are in bijection with $\Ho^1(\Gamma_x,G(k))$. 
\end{prop}

\begin{proof} It is enough to show that the evaluation map induces an injection of $\Ho^1(\Gamma_z,G(\Db_x))$ into $\Ho^1(\Gamma_x,G(k))$ since this map is already surjective. This can be seen as a consequence of either  \cite[Proposition 2.9]{damiolini:2021:equivariant} or \cite[Claim 3.2]{gille:2018:semisimple}. \end{proof} 

Using the same terminology introduced by \cite{balaji.seshadri:2015:moduli} and used also in \cite{damiolini:2021:equivariant}, we define the concept of local type in cohomological terms. 

\begin{defi}\label{def:localtype} For every $(\Gamma,G)$-bundle $\Ec$ on $C$ and $x \in C$ with stabilizer $\Gamma_x$, we call \textit{local type} of $\Ec$ at $x$ the element $\kappa \in \Ho^1(\Gamma_x,G):=\Ho^1(\Gamma_x,G(k))$ which corresponds to the isomorphism class of $\Ec$ on $\Db_x$ via \cref{prop:twgen}.\end{defi}

\begin{rmk} \label{rmk:localtypesgamma} Observe that for all $\gamma \in \Gamma$ and $x \in C$ one has that $\Gamma_{\gamma(x)} = \gamma \Gamma_x \gamma^{-1}$. Moreover, if an element $\kappa \in\Ho^1(\Gamma_x,G)$ describes the local type of the $(\Gamma,G)$-bundle $\Ec$ around $x$, this uniquely determines the element $\gamma(\kappa) \in\Ho^1(\Gamma_{\gamma(x)},G)$ corresponding to the local type of $\Ec$ around $\gamma(x)$. Thus, specifying the local type of $\Ec$ at $x$ automatically specifies the local type of $\Ec$ at all the points in the same orbit of $x$. \end{rmk}

\subsubsection{Patching} We can use the main theorem of  \cite{beauville.laszlo:descent}, to show that $(\Gamma,G)$-bundles on $C$ can be obtained by patching together local $(\Gamma,G)$-bundles. The covering that we will use consists of the complement of the ramification locus $R$ of the action of $\Gamma$ on $C$, and of formal neighborhoods $\Db_x$ about every point $x \in R$. We begin with the following result:

\begin{lemma} \label{lem:trivialGGamma} Let $\Ec$ be a $(\Gamma,G)$-bundle on $C$. Let $U = C \setminus R$ be the complement of the ramification locus of the action of $\Gamma$ on $C$. Then the restriction of $\Ec$ to $U$ is isomorphic to the trivial $(\Gamma,G)$-bundle over $U$. \end{lemma}

\begin{proof} In view of \cite[Theorem 3.1]{damiolini:2021:equivariant} it is enough to show that $\pi_*(\Ec)^\Gamma$ is the trivial $\pi_*(G\times U)^\Gamma$-bundle on the quotient curve $U/\Gamma$. Since U does not contain any ramification point, we deduce from \cite[Proposition 2.9]{damiolini:2021:equivariant} that $\pi_*(\Ec)^\Gamma$ is a $\pi_*(G\times U)^\Gamma$-bundle. Moreover, the group $\pi_*(G\times U)^\Gamma$ is a parahoric Bruhat--Tits group over the affine curve $U/\Gamma$, hence \cite[Theorem 1]{heinloth:2010:uniformization} guarantees that $\pi_*(\Ec)^\Gamma \cong \pi_*(G\times U)^\Gamma$, concluding the argument.
\end{proof}

\begin{rmk} The idea underlying the proof of the above result can be found also in the proof of  \cite[Lemma 3.1]{hong.kumar:2022}. \end{rmk}

The following result tells us that we can reconstruct $(\Gamma,G)$-bundles from local types at the ramification points. 

\begin{prop} \label{prop:nonempty}
Let $C$ be a smooth and projective curve with a $\Gamma$ action and $\rho \colon \Gamma \to \text{\rm Aut}(G)$ be a group homomorphism. Let $R \subset C$ be the ramification locus for the action of $\Gamma$. Choose $S=\{x_1, \dots, x_s\} \subseteq R$ such that $R = \sqcup_{x_i \in S} \Gamma x_i$ and, for every $x \in S$ choose an element $\kappa_{x} \in \Ho^1(\Gamma_{x}, G)$. Then there exists a $(\Gamma,G)$-bundle over $C$ with local types $\gamma(\kappa_{x})$ at every point $\gamma(x) \in R$ with $x \in S$ and $\gamma \in \Gamma$. 
\end{prop}

\begin{proof} In view of \cref{rmk:localtypesgamma} it is enough to show that there exists a $(\Gamma,G)$-bundle over $C$ with local types $\kappa_{x}$ at every point $x \in S$. From  \cref{prop:twgen} we know that $\kappa_{x}$ extends uniquely to an element of $\Ho^1(\Gamma_{x}, G(\Db_x))$, and this defines an isomorphism class of $(\Gamma,G)$-bundles on the disjoint union $\sqcup_{\gamma \in \Gamma/\Gamma_x} D_{\gamma(x)}$ of the prescribed local type. On the open curve $U=C \setminus R$, we take the trivial $(\Gamma,G)$-bundle $U \times G$. For these data to give rise to a $(\Gamma,G)$-bundle on the whole curve $C$, it is enough to show that $\Ho^1(\Gamma_x,G(\Db_x^\times))$ is trivial. Rewriting this in terms of Galois cohomology, this equals $\Ho^1(\Kc_{m_x}/\Kc, H)$, where $H:=(\Res_{\Kc_{m_x} / \Kc}(G_{ \Kc_{m_x}}))^{\Gamma_x}$ is a smooth and connected group over $\Spec(\Kc)$. Recall that $\Ho^1(\Kc_{m_x}/\Kc, H)$ can be identified with the set of isomorphism classes of $H$-bundles over $\Spec(\Kc)$ which are trivial after base change to $\Spec (\Kc_{m_x})$. Thus, the triviality of $\Ho^1(\Kc_{m_x}/\Kc, H)$ follows from the fact that the absolute Galois cohomology $\Ho^1(\Kc, H)$ vanishes, and this follows from \cite[page 484]{borel.springer:1968:rationalityII}.
\end{proof}

\subsection{Components of \texorpdfstring{$\Bun_{\Gamma,G}$}{Bun(Gamma,G)}} 
Let $C$ be a smooth and projective curve on which the finite group $\Gamma$ acts and assume that $\Gamma$ acts on the group $G$. We focus here on the moduli stack $\Bun_{\Gamma, G}$ of $(\Gamma, G)$-bundles on $C$. 

As in the statement of \cref{prop:nonempty}, let $S=\{x_1,\cdots, x_s\}$ be a subset of the ramification locus $R$ such that $R=\sqcup_{i=1}^s \Gamma x_i$. Further denote by $\Gamma_i$ the stabilizer groups at $x_i$.  For every $x_i \in S$, choose an element $\kappa_i\in \Ho^1(\Gamma_i, G)$ and denote by $\vec{\kappa}$ the $s$-tuple $(\kappa_1, \dots, \kappa_s)$. Denote by $\Bun_{\Gamma, G, \vec{\kappa}}$ the moduli stack of $(\Gamma, G)$-bundles on $C$ with local types $\vec{\kappa}$ at $\vec{x}$.  It follows from  \cref{prop:nonempty} that $\Bun_{\Gamma, G, \vec{\kappa}}$ is always non-empty.  Moreover, by \cite{damiolini:2021:equivariant},     $\Bun_{\Gamma, G, \vec{\kappa}}$ is isomorphic to $\Bun_{\mathcal{G}_{\vec{\kappa}} }$,  where  
\[\mathcal{G}_{\vec{\kappa}}={\rm Aut}_{\Gamma,G}(\Ec),\]
for a $(\Gamma, G)$-bundle $\Ec$ in $\Bun_{\Gamma, G, \vec{\kappa}}$.

 \begin{prop}\label{thm_comp}
 If $G$ is simply-connected, then 
 the group scheme $\mathcal{G}_{\vec{\kappa}}$ is a parahoric Bruhat--Tits group scheme over $\bar{C}=C/\Gamma$. In particular, $\Bun_{\Gamma, G, \vec{\kappa}}$ is connected.
 \end{prop}

 \begin{proof}
Let $\gamma_i$ be a generator of $\Gamma_i$. Then the local type $\kappa_i=[t_i]\in \Ho^1(\Gamma_i,G)$ gives rise to a finite order automorphism $\sigma_i:=\Ad_{t_i}\circ \gamma_i$ on $G$. By 
\cref{prop:twgen}, $\Ho^1(\Gamma_i, G)\cong \Ho^1(\Gamma_i, G(\mathbb{D}_{x_i}))$. 
Fix any $\Ec\in \Bun_{\Gamma,G, \vec{\kappa}}$. Then, $[\Ec|_{\mathbb{D}_{x_i}}]\in H^1(\Gamma_x, G(\mathbb{D}_x ))\cong H^1(\Gamma_x, G)$ (by \cref{prop:twgen}). It actually means that $\Ec|_{\mathbb{D}_{x_i}}$ is isomorphic to the $(\Gamma_x, G)$-bundle $\Ec^\circ_i:=\mathbb{D}_{x_i}\times G$, which is a trivial $G$-bundle and the action of  $\gamma_i$ is given by $(z, g)\mapsto (\gamma_i(z), t_i\gamma_i(g))$. This implies that 
\[ {\rm Aut }_{\Gamma_i,G}(\Ec|_{\mathbb{D}_{x_i} })\cong {\rm Aut}_{\Gamma_i,G}(\Ec^\circ_i)\cong {\rm Res}_{\mathbb{D}_{x_i}/\mathbb{D}_{\bar{x}_i} }( G_{\mathbb{D}_{x_i}} )^{\sigma_i},  \]
where $\bar{x}_i$ is the image of $x_i$ under the projection map $C \to C/\Gamma$. 
By 
\cref{cor:parahoric}, ${\rm Aut }_{\Gamma_i,G}(\Ec|_{\mathbb{D}_{x_i} })$ is a parahoric group scheme. Thus, $\mathcal{G}_{\vec{\kappa}}={\rm Aut}_{\Gamma,G}(\Ec)$ is a parahoric Bruhat--Tits group scheme over $\bar{C}$. 

The second statement follows from  \cite[Theorem 2]{heinloth:2010:uniformization}.
\end{proof}

The following result can be seen as a consequence of the above proposition together with \cite[Theorem 4.2]{Pappas-Rapoport:2022} or directly from \cite[Corollary 7.2]{Pappas-Rapoport:2022}.

\begin{cor} \label{cor:comp} The stack $\Bun_{\Gamma,G}$ is decomposed  as the disjoint union of connected  algebraic stacks  $\sqcup_{\vec{\kappa}} \Bun_{\Gamma,G,\vec{\kappa}}$.
\end{cor}

We now use some of the consequences of \cref{thm:H1tauG} established in \cref{appendix} to compute the number of connected components of $\Bun_{\Gamma,G}$ in a variety of examples.

\begin{eg}
We now assume that $\Gamma=\langle \gamma \rangle$ is a cyclic group of order $r$ and its generator acts on $G$ by a diagram automorphism of order $r$ as well. Let $s$ be the number of ramified points for the action of $\Gamma$ on $C$. We combine  \cref{cor:comp} and \cref{eg:diagram} to show that
\begin{itemize}
    \item if $(G,r)  \in \{(A_{2\ell-1},2), (D_{\ell+1},2), (E_6,2), (D_4,3) \}$ then $\Bun_{\Gamma,G}$ has $2^s$ connected components;
    \item if $(G,r) = (A_{2\ell},2)$, then $\Bun_{\Gamma,G}$ has only one component.
\end{itemize}
Finally, when $\Gamma$ is trivial, then $\Bun_{\Gamma,G}=\Bun_{G}$, which is connected. This can also be seen directly from \cref{cor:comp} and \cref{lemma:alcove}.\end{eg}

\begin{eg} Assume that the group $\Gamma\cong \ZZ/2\ZZ$ acts trivially on $G=\SL_{n}$. Let $s$ be the number of ramified points for the action of $\Gamma$ on $C$. Then we deduce from \cref{cor:comp} and \cref{eg:tautrivial} that the stack $\Bun_{\ZZ/2\ZZ,\SL_{n}}$ has $\left\lceil \frac{n+1}{2}\right\rceil^s$ connected components.
\end{eg}

\begin{eg} We can construct explicitly a $(\Gamma, G)$-bundle which has different local types than the trivial $(\Gamma, G)$-bundle.  Let $\Ec$ be the following $(\Gamma, G)$-bundle:  $\Ec=C\times G$. Fix a $\gamma$-invariant element $t \neq 1$ in $G$ such that $t^2=1$.  $G$ acts on $\Ec$ on the right, and $\gamma(p, x)= (\gamma(p), t\gamma(x))$.  Then, the local type of $\Ec$ at a ramified point is given by $[t]\in \Ho^1(\Gamma, G)$. When $G$ is of type $A_{2\ell-1}$, we can find an element $t$ such that $[t]$ is not trivial. For example,  we can take $t=\alpha^\vee_\ell(-1)$, where $\alpha^\vee_\ell$ is the $\ell$-th simple coroot, regarded as a cocharacter of the maximal torus $T$. 
 \end{eg}

\section{Parahoric Bruhat--Tits group schemes arising from coverings} \label{sec:BTfromcoverings}

\subsection{Reductive group schemes over curves}

In this section we denote by $\Gc$ a smooth and simple group scheme over a smooth curve $C$. i.e. the group $\Gc$ is smooth over $C$ and every geometric fiber is a simple group. Let $G$ be the simple group over $k$ with the same root datum as one geometric fiber of $\Gc$ and let $G_{C}$ denote its pullback to $C$.  Let $\Ec:=\Iso(G_{C}, \Gc)$ denote the scheme over $C$ classifying the group scheme isomorphisms from $G_{C}$ to $\Gc$ over $C$. In view of the following result, $\Ec$ is a right ${\rm Aut}(G)$-bundle. 

\begin{lemma}
\label{lem:equivSGA}\cite[Corollaire 1.17, Exposé XXIV]{SGA3:volIII}
There exists an equivalence between the category $\mathrm{Form}(G)_C$ of forms of $G$ over $C$, and the category   ${\Bun}_{{\Aut}(G), C} $ of ${\rm Aut}(G)$-bundles over $C$, given by $G' \mapsto {\Iso}(G_C, G')$.  The inverse is given by $\Fc \mapsto G_\Fc:= \Fc \times^{\Aut(G)}G$.
\end{lemma}

We will also make use of the following result.

\begin{lemma}\label{lem:transitiveaction}
Let $H$ be a finite group acting on a variety $X$ over $k$. Suppose that $H$ permutes the set of connected components of $X$ transitively. Let $X^\circ$ be a component of $X$ and $H^\circ$ be the stabilizer of $X^\circ$ in $H$. Then, the morphism $\phi \colon  H\times^{H^\circ}  X^\circ\to X$ given by $(h, x)\mapsto h\cdot x$, is an $H$-equivariant isomorphism. 
\end{lemma}

\begin{proof}
Since $H$ permutes the component set of $X$ transitively,  the morphism $\phi$ must be surjective. Since for any $h\in H$, the map $h \colon X^\circ\to h X^\circ$ is an isomorphism, then $\phi$ is also a closed immersion.  Thus, $\phi$ is an isomorphism of varieties. 
\end{proof}

Let $\Eco$ denote a connected component of $\Ec$ and denote by $\Ct$ the quotient  $\Eco/G_{\ad}$. Since $G_{\ad}$ is connected, $\Ct$ coincides with a component of the quotient $\Ec / G_{\ad} $. Moreover $\Ct$ is a connected \'etale covering over $C$ and $\Ec/G_{\ad}$ is a right ${\rm Out}(G)$-bundle. We define $\Gamma$ to be the subgroup of ${\rm Out}(G)$ which stabilizes the component $\Ct$. It then follows from \cref{lem:transitiveaction} that $\pi \colon  \Ct \to C$ is an étale $\Gamma$-covering, that is the map $\pi \colon \Ct \to C$ is a $\Gamma$-principal bundle and $\Ct/\Gamma =C$. We summarize this construction and notation in the diagram below.

\[
  \begin{tikzcd}
  \Gc \arrow[d] & \Ec \arrow{dl}[near start]{{\small \Aut_G\text{-bundle}}} \arrow[r, hookleftarrow]& \Eco \arrow{d}{\small G_{\ad}\text{-bundle}} \\
  C \arrow[leftarrow]{rr}{\pi} &&\Ct=\Eco/G_\ad 
\end{tikzcd}
\]

We recall the following fact from étale descent along principal bundles. 
\begin{lemma} \cite[Theorem 4.46]{vistoli:descent}
\label{lem:descentG}
Let $\pi\colon \Ct \to C$ be an \'etale $\Gamma$-covering of algebraic curves for some finite group $\Gamma$. 
There is an equivalence between the category of smooth group schemes over $C$ and smooth $\Gamma$-group schemes over the $\Gamma$-curve $\Ct$, where the pullback functor $\pi^*$ and the $\Gamma$-invariant direct image $\pi_*^\Gamma$ are inverse to each other. \end{lemma}

To be more explicit, we observe that if $\Gc$ is a $\Gamma$-group scheme over $\Ct$, then by adjunction there is a map $\pi^* \pi_*(\Gc) \to (\Gc)$ which induces  
\begin{equation} \label{eq:unit} \epsilon \colon \pi^* \pi_*^\Gamma(\Gc) \to \pi^* \pi_*(\Gc) \to \Gc.\end{equation} Since $\pi$ is étale, then $\epsilon$ is an isomorphism. Similarly, for every group $\Hc$ over $C$, the map $\pi$ induces 
\begin{equation} \label{eq:counit} \eta \colon \Hc \to \pi_*(\pi^*\Hc)^\Gamma\end{equation} which is an isomorphism as well.

 We now fix a splitting
  \begin{equation}
  \label{eq:splitting}
   \iota: {\rm Out}(G)\to  {\rm Aut}(G)\end{equation}
   preserving a Borel subgroup $B$ and a maximal torus $T$ contained in $B$.  This gives rise to a group homomorphism $\phi_\iota \colon \Gamma \to {\rm Aut}(G)$.  It then follows that $\Eco$ is a $(\Gamma,G_{\ad})$-bundle over $\Ct$, where the action of $\Gamma$ is on the right and the action of $G_\ad$ is on the left.
  
   We now denote by $G_{\Eco}$ the group scheme over $\Ct$ associated to the $G_\ad$-bundle $\Eco$, namely $G_{\Eco}:= \Eco\times^{G_{\ad}} G $, where $G_{\ad}$ acts on $G$ by conjugation. The group scheme $G_{\Eco}$, together with the induced $\Gamma$-action, is a group scheme over $\Ct$ which is equipped with an action of $\Gamma$ compatible with that on $\Ct$.

\begin{lemma}
\label{lem:reductive} Under the above assumptions, there exists a natural isomorphism of group schemes over $C$
\[   \Gc \cong \pi_*( G_{\Eco} )^\Gamma. \]
\end{lemma}

\begin{proof} The map $\Ct \to C$ is a torsor for the group $\Gamma$, hence by \cref{lem:descentG} it is enough to check that $\pi^*\Gc \cong G_{\Eco}$.

Since $\Ct=\Eco/G_{\ad}$, this can be further reduced to show that there exists an $G_{\ad}\rtimes \Gamma$-equivariant isomorphism between  $\Eco \times_k G$ and $\Eco \times_{C} \Gc$.  To establish such an isomorphism, we first construct an ${\rm Aut}(G)$-equivariant isomorphism between $\Ec \times_{C} G_{C}$  and $\Ec \times_{C} \Gc$, where the action of ${\rm Aut}(G)$  is given by
\[ \Ec \times_{C} G_{C} \times  {\rm Aut}(G)  \to \Ec \times_{C} G_{C}, \qquad ((\phi, g) , \psi) \mapsto (\phi \circ \psi, \psi^{-1}(g))
\]
\[\Ec \times_{C} \Gc  \times {\rm Aut}(G)  \to \Ec \times_{C} \Gc , \qquad ((\phi, \alpha) , \psi) \mapsto (\phi \circ \psi, \alpha).
\]
It follows that the map
\[ f \colon  \Ec \times_{C} G_{C}  \to  \Ec \times_{C} \Gc , \qquad (\phi,g) \mapsto (\phi, \phi(g)) \] is ${\rm Aut}(G)$-equivariant and it is also an isomorphism which further restricts to an isomorphism 
 \[ \Eco \times_C G_C  \cong  \Eco \times_{C} \Gc ,\]
which is $G_{\ad}\rtimes \Gamma$-equivariant.  This concludes the lemma.  \end{proof}

\begin{prop} \label{prop:CaffineGtrivial}
If the smooth algebraic curve $C$ is affine, then $ \mathcal{G} \cong \pi_*( G \times \Ct  )^\Gamma $, where the action of $\Gamma$ on $G_{\Ct}$ is induced from the splitting \eqref{eq:splitting} and $\Gamma$ acts on $\Ct$ described above. 
\end{prop}

\begin{proof}
\Cref{lem:trivialGGamma} shows that the $(\Gamma, G_\ad)$-bundle ${\Eco}$ is isomorphic to the trivial $(\Gamma,G_\ad)$-bundle $\Ct \times G_\ad$. Then, the proposition follows from \cref{lem:reductive}. 
\end{proof}

We conclude this section with a description of generically split groups over $C$. For this purpose, we recall that a simple group scheme $\mathcal{G}$ over $C$ is called generically split if $\mathcal{G}_{k(C)}$ is a split simple group scheme over the function field $k(C)$ of $C$. 

\begin{prop} 
\label{prop:genericallyreductive}
The simple group scheme $\mathcal{G}$ is generically split if and only if $\mathcal{G}\cong G_{\Fc}:=\Fc \times^{G_{\ad}} G$ for some $G_{\ad}$-bundle $\Fc$ on $C$.  In particular, if $C$ is affine, then $\mathcal{G}$ is split generically if and only if $\mathcal{G}\cong G \times C$.  
\end{prop}

We note that when $C$ is affine, this result can also be found in \cite[Corollary 3.2.]{chernousov.gille.pianzola:2016:three}. 

\begin{proof}
Suppose that $\mathcal{G}\cong G_{\Fc}$ for some $G_{\ad}$-bundle $\Fc$ over $C$. Since generically $\Fc$ is trivializable, $\mathcal{G}$ is generically split.  

Conversely, by \cref{lem:reductive}  we know that  $\mathcal{G}\cong \pi_*( G_{\Eco}  )^\Gamma$. It is enough to show that if $\mathcal{G}$ is generically split, then necessarily $\Gamma$ must be trivial. Let $K$ (resp. $L$) be the function field of $C$ (resp. $\Ct$), so that ${\rm Gal}(L/K)=\Gamma$. It suffices to show that when $\Gamma$ is nontrivial,  $\mathcal{G}_K={\Res}_{L/K}(G_L)^\Gamma$ is not split.  The $K$-form $\mathcal{G}_K$ of $G_L$ gives a class $\kappa$ in $\Ho^1(\Gamma, {\rm Aut}(G)(L) )$. The pointed set  $\Ho^1(\Gamma, {\rm Aut}(G)(L) )$  classifies the isomorphism classes of all $K$-forms of $G_L$. Recall that $\Ho^1(\Gamma, {\rm Out}(G) )$ can be identified with the isomorphisms classes of all group homomorphisms from $\Gamma$ to ${\rm Out}(G)$.  With respect to the splitting $\iota: {\rm Out}(G) \to {\rm Aut}(G)$,  it induces a map $\Ho^1(\Gamma, {\rm Out}(G) )\to \Ho^1(\Gamma, {\rm Aut}(G)(L))$.  This map is injective, since the composition of ${\rm Out}(G)\to {\rm Aut}(G)(L)\to{\rm Out}(G) $ is the identity map.  We also observe that, the group homomorphism $u: \Gamma\to {\rm Out}(G)$ induced from the action of $\Gamma$ on $G$, corresponds to $\kappa$. Since $u$ is nontrivial, it follows that $\kappa$ is also nontrivial. Thus, $\mathcal{G}_K$ is not split. 

When $C$ is affine, $\Ec$ is trivializable.  Thus, the second statement also follows. 
\end{proof}

\subsection{Parahoric Bruhat--Tits group schemes over curves} 
Suppose that $C$ is a connected projective smooth curve over $k$. In this section we will further assume that $\car(k)=0$, assumption needed for \cref{thm:cover}.  Let $\pi\colon \Ct\to C$ be a $\Gamma$-covering of $C$. For any $x\in C$,  let $m_x$ denotes the \textit{ramification index} at $x$; equivalently for any $\tilde{x}\in \Ct$ such that $\pi(\tilde{x})=x$, we have $m_x=|\Gamma_{\tilde{x}}|$. 

\begin{lemma}
\label{thm:cover}
Let $C$ be any connected projective smooth curve $C$ over $k$. Let $g$ be the genus of $C$. There exists a connected $\Gamma$-covering $\pi\colon \hat{C}\to C$ of $C$  for some finite group $\Gamma$, with prescribed branched points $x_1,\cdots, x_s\in C$ and prescribed ramification indices $m_1,\cdots, m_s$ at $x_1,\cdots, x_s$, except the following cases:
\begin{enumerate}
\item $g=0$, $s=1$;
\item $g=0$, $s=2$ and $m_1\not=m_2$.
\end{enumerate}
\end{lemma}
\begin{proof}
Using \cite[Theorem 6.4.2]{serre:1992}, the proof proceeds as in the proof of \cite[Lemma 2.5]{poonen:2005}.
\end{proof}

Given a parahoric Bruhat--Tits group scheme $\mathcal{G}$ over $C$, we say that $\mathcal{G}$ is ramified at $x\in C$ if the fiber $\mathcal{G}_x$ at $x$ is not reductive.  If $x\in C$ is a ramified point of $\mathcal{G}$, then $\mathcal{G}|_{ \mathbb{D}_x }$ is a parahoric group scheme over $ \mathbb{D}_x$.

\begin{theorem}\label{thm:parahoricfromcoverings}
Let $\mathcal{G}$ be any parahoric Bruhat--Tits group scheme over a connected smooth projective curve $C$ of genus $g$ with $s$---possibly zero---ramified points in $C$. Assume that $\Gc|_{k(C)}$ is simply connected and let $G$ be the split semisimple and simply connected group over $k$ having the absolute type of $\Gc|_{k(C)}$. 
Suppose further that $(g, s) \neq (0, 1)$. Then, there exists a $\Gamma$-covering $\pi \colon \Ct\to C$ of $C$ for some finite group $\Gamma$ and a $(\Gamma, G_{\ad})$-bundle $\Ec$ where $\Gamma$ acts on $G_{\ad}$ by diagram automorphisms, such that $\mathcal{G}\cong  \pi_*(G_{\Ec})^\Gamma$. 
\end{theorem}

\begin{proof}[Proof of \cref{thm:parahoricfromcoverings}]
    Since the case $s=0$ is covered by \cref{lem:reductive}, we can assume that $s\geq 1$.

    We will begin with some notation. We first of all denote by $R=\{x_1,x_2,\cdots, x_s\}$ the set of all ramified points of $\mathcal{G}$ and by $S$ the set $\{1, \dots, s\}$. The affine curve $C\setminus R$ is denoted by $C^\circ$. For every $i \in S$, let $z_i$ be a formal parameter around $x_i$, and let $u_i \colon \Db \to \Db_{x_i} \subset C$ be the isomorphism identifying $z$ with $z_i$.

    By \cref{prop:CaffineGtrivial},   there exists an isomorphism $\mathcal{G}|_{C^\circ}\cong {\Res}_{\hat{C}^\circ / C^\circ }(G_{\hat{C}^\circ})^D$, for some \'etale $D$-covering $\hat{C}^\circ$ of $C^\circ$, and a faithful action of $D$ on $G$ preserving a triple $(B,T, e)$, where $T\subset B$ and $e$ is a pinning of the pair $(B,T)$.  Then, there exists an isomorphism $ u_i^*(\mathcal{G})|_{\Db^\times}\cong {\Res}_{\Db^\times_{r_i}/ \Db^\times  } ( G_{ \Db^\times}  ) ^{\tau_i}$,  with $\tau_i\in D$.  
 
    Since by definition $u_i^*\Gc$ is a parahoric group scheme, \cref{thm:BT}  tell us that there exists an isomorphism 
    \begin{equation}\label{local_isom_1}
    \phi_i\colon  u_i^*\mathcal{G}\cong  {\Res}_{\Db_{m_i}/\Db} (G_{\Db_{m_i}  })^{\sigma_i}, \end{equation}
    where $\sigma_i$ acts on $G$ by $\tau_i\circ {\rm Ad}_{t_i}$ for some $t_i\in T^{\tau_i}$, and $\sigma_i$ acts on $\Db_{m_i}$ as usual.  Note that, by  \cref{thm:local},  the isomorphism $\phi_i$ always exists if we replace $m_i$ by any of its multiples and  take the same action of $\sigma_i$ on $G$. In particular, when $g=0$ and $s=2$, we can assume that $m_1=m_2$, hence we can assume to be in the hypothesis of \cref{thm:cover}.

    By \cref{thm:cover}, there exists a  $\Gamma'$-covering $\qp \colon  \XC\to C$ of $C$ for some finite group $\Gamma'$, such that the ramification indices at each $x_i$ is $m_i$.    Then $\qp^*\mathcal{G}$ is a smooth group scheme over $\XC$. For each $x_i$,  choose a point $x'_i$ above $x_i$ and let $\Gamma'_i$ be the stabilizer of $x'_i$ in $\Gamma'$. Let $\Db_{x_i}$ (resp. $\Db_{x'_i}$) be the formal neighborhood of $x_i$ (resp. $x'_i$) in $C$ (resp. $\XC$).    Then the isomorphism \eqref{local_isom_1} can be reinterpreted as the following isomorphism

\begin{equation}\label{local_isom_2}
 \phi_i\colon   \mathcal{G}|_{\Db_{x_i} }\cong  {\Res}_{\Db_{x'_i}/ \Db_{x_i}}   (G_{\Db_{x'_i} }  )^{\Gamma'_i} ,\end{equation}
    where $\Gamma'_i=\langle \gamma_i \rangle  $ with $\gamma_i$ being of order $m_i$ and   acting on $G$ by $\sigma_i$. Then the isomorphism \eqref{local_isom_2} together with the natural isomorphism ${\Res}_{\Db_{x'_i}/ \Db_{x_i}}   (G_{\Db_{x'_i} }  )^{\Gamma'_i}\cong f^{\Gamma'}_*(G_{f^{-1}(\Db_{x_i}) }) $ give rise to the following isomorphism: 
    \[  \psi_i:  \mathcal{G}|_{\Db_{x_i} }\cong f^{\Gamma'}_*(G_{f^{-1}(\Db_{x_i}) }) ,   \]
    where $f^{-1}(\Db_{x_i})= \bigcup_{\gamma\in \Gamma'/\Gamma'_i} \Db_{\gamma \cdot x'_i}  $, and $G_{f^{-1}(\Db_{x_i})  }$ is the constant group scheme over $f^{-1}(\Db_{x_i})$. 

    Recall the maps $\epsilon$ and $\eta$ from \eqref{eq:unit} and \eqref{eq:counit}. 
    For each $1\leq i\leq s$,  the map $\epsilon$ induces a natural morphism of group schemes
    \[  \epsilon_i\colon  \qp^* ( \qp^{\Gamma'}_* (G_{f^{-1}(\Db_{x_i}) })    ) \to    G_{f^{-1}(\Db_{x_i}) }  , \]
    which is $\Gamma'$-equivariant and is an isomorphism over $f^{-1}(\Db^\times_{x_i})$, since $f$ is étale on $\Db^\times_{x_i}$. 
    There exists a unique smooth $\Gamma'$-group scheme  $\mathcal{G}'$ over $\XC$ such that $\mathcal{G}'|_{f^{-1}(C^\circ)}=\qp^*(\mathcal{G}|_{C^\circ})$, and for each $i$ an isomorphism $\psi'_i\colon \mathcal{G'}|_{f^{-1}(\Db_{x_i}) }\cong G_{f^{-1}(\Db_{x_i}) }$ of $\Gamma'$-group schemes such that 
    \[\psi_i' =\epsilon_i \circ  \qp^*(\psi_i)    .\]
    We now show that there is an isomorphism 
     \begin{equation}\label{eq:des1}
       \mathcal{G} \cong \qp^{\Gamma'}_*(\mathcal{G}' ). \end{equation}

    The group scheme $\mathcal{G}'$ is glued from $\qp^*(\mathcal{G}|_{C^\circ})$ and $\{ G_{f^{-1}(\Db_{x_i}) }  \}_{i=1,\cdots, s}$ via the transition isomorphisms $\{ \psi'_i \}_{i=1,\cdots, s}$.  By \cref{lem:descentG}, the map $\eta$ gives rise to an isomorphism 
\[ \eta_{C^\circ} \colon  \mathcal{G}|_{C^\circ}\cong  \qp^{\Gamma'}_*(f^*(\mathcal{G}_{C^\circ} )). \]
    Moreover, we also have isomorphisms $\psi_i \colon   \mathcal{G}|_{\Db_{x_i} }\cong f^{\Gamma'}_*(G_{f^{-1}(\Db_{x_i}) })$. We shall show that $\eta_{C^\circ}$ and $\{ \psi_i  \}_{i=1,\cdots,s}$ can glue to an isomorphism $\mathcal{G} \cong \qp^{\Gamma'}_*(\mathcal{G}' )$.  It suffices to show the commutativity of the following diagram:
\begin{equation}\label{main_thm_diagram}
\begin{tikzcd}
\mathcal{G}_{\Db^\times_{x_i}} \arrow{d}{\eta_{ \mathcal{G}|_{\Db^\times_{x_i}}  }} \arrow[rr, "\psi_i"] && f^{\Gamma'}_*(G_{f^{-1}(\Db_{x_i}) })  \\
f_*^{\Gamma'}\circ f^* (\mathcal{G}|_{\Db^\times_{x_i}} )  \arrow[rr,"f_*^{\Gamma'}\circ f^* (\psi_i)"  ] && f_*^{\Gamma'}\circ f^*\circ f_*^{\Gamma'} (G_{f^{-1}(\Db_{x_i}) }))  \arrow[u,"f_*^{\Gamma'}(\epsilon_i) "]
\end{tikzcd} ,\end{equation}
    where $f_*^{\Gamma'}(\epsilon_i)\circ f_*^{\Gamma'}\circ f^* (\psi_i)=f_*^{\Gamma'}(\psi'_i ) $. By the adjunction between $f^{\Gamma'}_*$ and $f^*$,  we have $f_*^{\Gamma'}(\epsilon_i)^{-1}= \eta_{f_*^{\Gamma'} (G_{f^{-1}(\Db_{x_i}) })  }$. Then, the commutativity of \eqref{main_thm_diagram} follows from the functoriality of the morphism ${\rm Id}\to f_*^{\Gamma'}\circ f^* $ of functors, applying to $\psi_i$. 

    Now, we proceed with a construction similar to that of \cref{lem:reductive}.  The scheme $\Ec= {\rm Iso}(G_{\XC}, \mathcal{G}' )$ is an ${\rm Aut}(G) $-bundle over $\XC$, with a commuting action of $\Gamma'$. Let $\Eco$ be a component of  $\Ec$.   The finite group ${\rm Out}(G)\times \Gamma'$ acts on $\Ec$.  Let $\Gamma$ be the subgroup of ${\rm Out}(G)\times \Gamma'$ which stabilizes the component $\Eco$ of $\Ec$.  Set $\Ct= \Eco/{G_{\ad}}$.  Then, $\pi \colon \Ct\to C$ is a $\Gamma$-covering of $C$, and $\Eco$ is a $(\Gamma, G_{\ad} )$-bundle over $\Ct$.

    Let $\XC^\circ$ (resp. $\Ct^\circ$) be the open unramified part of $\XC$ (resp. $\Ct$) with respect to the $\Gamma'$ (resp. $\Gamma$)-action.  Recall that $\XC= \Ec/ { {\rm Aut}(G)   }$ and ${\rm Aut}(G) =G_{\ad}\rtimes {\rm Out}(G)$. Let $D'$ be the stabilizer of the component $\Eco$ in ${\rm Out}(G)$ via the splitting \eqref{eq:splitting}. By \cref{lem:transitiveaction}, $\XC= \Eco/{(G_{\ad}\rtimes D' )}$.  Thus, the natural morphism $\QP\colon \Ct\to \XC$ is an \'etale $D'$-covering of $\XC$. Moreover, the projection map $\Gamma\to \Gamma'$ descends to an isomorphism of groups  $\Gamma/D'\cong \Gamma'$. 
 
    Finally, we are ready to verify that $\mathcal{G}\cong  \pi^\Gamma_*(G_{ \Eco })$.  Note that $\qp_* ^{{\Gamma'}} \circ  \QP_*^{ D'} \cong \pi_*^\Gamma$. By \eqref{eq:des1},  it suffices to check that $\QP_*(G_{\Eco})^{D'} \cong \mathcal{G}'$.   Since $\QP$ is \'etale, by \cref{lem:descentG},  it is enough to show that $G_{\Eco}\cong  q^*\mathcal{G}'$ as $\Gamma$-group schemes. 
    In other words, it suffices to show that 
\[ G_{\Eco}\cong \Ct\times_{\XC} (\Ec\times ^{{\rm Aut}(G) } G)   . \]
    By \cref{lem:transitiveaction},  $ \Ec\times ^{{\rm Aut}(G) } G\cong \Eco\times ^{G_{\ad}\rtimes D'} G$.  
    Thus, we are left to show that 
\[ G_{\Eco}\cong    \Ct\times_{\XC} ( G_{\Eco} / D'),\]
    which holds true, since  $\QP\colon \Ct\to \XC$ is \'etale. This concludes the proof of our theorem.  \end{proof}


\appendix

\section{Local types via alcoves} \label{appendix}

In this appendix we give a concrete realization of the non-abelian group cohomology $\Ho^1(\Gamma,G)$ as a consequence of \cref{thm:H1tauG}. As an application, we obtain an explicit classification of finite order automorphisms of $G$ which enhance the analogue classification provided by \cite[Theorem 8.6]{kac:1990:infinite} at the level of (twisted) affine Lie algebras, to positive characteristic (see \cref{thm:clas}). Along the way We also provide many concrete examples.  

One motivation to give a precise description of these spaces arises from the following geometric question: How to effectively describe $(\Gamma,G)$-bundles on a curve? In \cref{sec:localtypes}, we have shown that the local description of a $(\Gamma,G)$-bundle is determined by its local type, and that this is naturally identified with an element of $\Ho^1(\Gamma,G)$.

As in previous sections, throughout we assume that $G$ is a simple group over an algebraically closed field $k$. We further assume that $\Gamma$ is a cyclic group (with generator $\gamma$) acting on $G$, and that $\car(k)$ does not divide the order of $\Gamma$. 

\subsection{Diagram automorphisms} 
We will begin by recalling, via a pictorial description given in \cref{eg:tabledynkin}, the diagram automorphisms of simply laced groups. In particular we can see that only $D_4$ admits an automorphism of order $3$ and that none of the simple roots of $A_{2\ell}$ are $\Gamma$-invariant. 
 
\begin{table}[ht]\begin{tabular}{c|c}
     Type & Diagram automorphisms  \\ \hline
     $A_{2\ell -1}$&
     \dynkin[
edge length=.75cm,
labels*={\alpha_1,\alpha_2, \alpha_{\ell}, \alpha_{2\ell-2},\alpha_{2\ell -1}},
involution/.style={blue!50,<->},
involutions={15;24}]A{**.*.**} \\
$A_{2\ell}$& \dynkin[
edge length=.75cm,
labels*={\alpha_1,\alpha_2, \alpha_{\ell}, \alpha_{\ell+1},\alpha_{2\ell-1}, \alpha_{2\ell}},
involution/.style={blue!50,<->},
involutions={16;25;34}]A{**.**.**} \\
     $D_{\ell +1 }$& \dynkin[
edge length=.75cm,
labels*={\alpha_1,\alpha_2, ,\qquad \alpha_{\ell -1},\alpha_{\ell}, \alpha_{\ell+1}},
involution/.style={blue!50,<->},
involutions={65}]D{}  \\ 
$D_{4}$& {\begin{tikzpicture}[baseline]
     \dynkin[edge length=0.75cm]{D}{4} 
 \draw [blue!50,->](root 4) edge[bend right=-80]  (root 1);
 \draw [blue!50,->](root 1) edge[bend right=-80]  (root 3);
 \draw [blue!50,->](root 3) edge[bend right=-80]  (root 4);
  \node [below right] at (root 4) {{\tiny $\alpha_4$}};
    \node [left] at (root 1) {{\tiny $\alpha_1$}};
      \node [right] at (root 2) {{\tiny $\alpha_2$}};
        \node [above right] at (root 3) {{\tiny $\alpha_3$}};
      \end{tikzpicture}} \\
     $E_6$ & \dynkin[
edge length=.75cm,
labels*={\alpha_1,\alpha_6, \alpha_2, \alpha_3,\alpha_4, \alpha_5},
involution/.style={blue!50,<->},
involutions={16;35}]E6 
\end{tabular} 
\caption{Diagram automorphisms of simply laced groups. \label{eg:tabledynkin}}
 \end{table}

\begin{eg} \label{eg:H1Tsc} Let $G$ be a simply connected group, so that $\coX(T)$ has a basis of coroots, and let $T$ be a $\tau$-invariant torus, so that we can apply \cref{prop:H1Tcomb}. When $\tau$ is the trivial diagram automorphism, then it follows that $\coX(T)^\tau =\coX(T) = \norm_\tau(\coX(T))$, so that $\Ho^1_\tau(\Gamma,T)$ is isomorphic to $(\ZZ/m\ZZ)^{\oplus \ell}$, where $\ell$ is the rank of $G$ and $m = |\Gamma|$. One could obtain this result directly from the definition of $\Ho^1_\tau(\Gamma,T)$, which is in bijection with $T_m$, the $m$-torsion of the torus $T$. Assume now that $\tau$ is not trivial. In view of \cref{eg:tabledynkin}, we can explicitly spell out the generators of $\coX(T)^\tau$ and $\norm_\tau \coX(T)$.

\begin{center}
\begin{longtable}{|c|c|l|l| }
     \hline Order of $\tau$ & Type & $\ZZ$-generators of $\coX(T)^\tau$ & $\ZZ$-generators of $\norm_\tau \coX(T)$\\ \hline
     
      \multirow{7}{*}{Order 2} & \multirow{2}{*}{$A_{2\ell-1}$} &  $\calpha_\ell$ & $2 \calpha_\ell$ \\
      && $\calpha_i + \calpha_{2\ell -i}$  for $i \in \{1, \dots, \ell-1\}$
     & $ ( \calpha_i + \calpha_{2\ell -i})$  for $i \in \{1, \dots, \ell-1\}$\\ \cline{2-4}
     
     & $A_{2\ell}$ &  $ \calpha_i + \calpha_{2\ell +1 -i}$  for $i \in \{1, \dots, \ell-1\}$ & $(\calpha_i + \calpha_{2\ell +1-i})$  for $i \in \{1, \dots, \ell\}$\\  \cline{2-4}
     
     &\multirow{2}{*}{$D_{\ell +1 }$} &  $\calpha_i$ for $i \in \{1,\dots, \ell-1\}$ & $2 \calpha_i$ for $i \in \{1,\dots, \ell-1\}$\\
     && $\calpha_{\ell}+\calpha_{\ell+1}$ & $(\calpha_{\ell}+\calpha_{\ell+1})$ \\  \cline{2-4}

     & \multirow{2}{*}{$E_{6}$} &  $\calpha_3$, $\calpha_6$ & $2 \calpha_3$, $2 \calpha_6$,\\
     & & $\calpha_1+\calpha_5$, $\calpha_2+\calpha_4$ 
     & $(\calpha_1+\calpha_5)$,  $(\calpha_2+\calpha_4)$\\  \hline
     
     \multirow{2}{*}{Order $3$} & \multirow{2}{*}{$D_4$} & $\calpha_2$ & $3 \calpha_2$ \\
     && $(\calpha_1+ \calpha_3+ \calpha_4)$ & $(\calpha_1+ \calpha_3+ \calpha_4)$\\ \hline
\end{longtable}\end{center}

Hence we can explicitly compute $\Ho^1_\tau(\Gamma,T)$ in this situation.

\begin{center}\renewcommand{\arraystretch}{1.6}
\begin{longtable}{|c|c|c|}
     \hline Order of $\tau$ & Type & $\Ho^1_\tau(\Gamma,T)$ \\ \hline
    \multirow{4}{*}{Order 2} & $A_{2\ell -1}$  & $\frac{\ZZ}{m\ZZ} \times \left(\frac{\ZZ}{(m/2)\ZZ}\right)^{\ell -1}$\\ \cline{2-3}
    &$A_{2\ell}$ & $\left( \frac{\ZZ}{(m/2)\ZZ}\right)^{\ell}$\\ \cline{2-3}
    &$D_{\ell +1}$ & $\left( \frac{\ZZ}{m \ZZ}\right)^{\ell-1} \times \frac{\ZZ}{(m/2)\ZZ}$\\\cline{2-3}
    &$E_6$ & $\left(\frac{\ZZ}{m\ZZ}\right)^{2} \times \left(\frac{\ZZ}{(m/2)\ZZ}\right)^{2}$\\  \hline
    Order 3 &$D_4$ & $\frac{\ZZ}{m \ZZ} \times \frac{\ZZ}{(m/3)\ZZ}$\\ \hline
\end{longtable}\end{center}
In particular, if we assume that $\tau$ has order $2$ and $m=2$, then we can see that $\Ho^1_\tau(\Gamma,T)$ is trivial in type $A_{2\ell}$ for $\ell \geq 1$, while it consists of two elements for type $A_{2\ell-1}$ with $\ell \geq 2$. From this, and \cref{lem:TsurgG}, we can immediately deduce that $\Ho^1_\tau(\Gamma,\SL_{2\ell+1})$ is trivial for $\ell \geq 1$. We can actually show that $\Ho^1_\tau(\Gamma,\SL_{2\ell})$ has indeed two elements for $\ell \geq 2$.
Let $B$ denote a Borel of $G = \SL_{2\ell}$ containing $T$ on which $\tau$ acts. One obtains that the map $\phi$
\[ \xymatrix{ (G/B)^\Gamma \ar[r]^-\phi &  \Ho^1_\tau(\Gamma, B) \ar[r]^-\psi & \Ho^1_\tau(\Gamma, G) \ar[r] &0,}
\] is zero.  
Since the map $\psi$ is surjective and  $|\Ho^1_\tau(\Gamma, B)|=|\Ho^1_\tau(\Gamma, T)|=2$, then also $|\Ho^1_\tau(\Gamma, G)|= 2$, concluding the example. 
\end{eg}

\begin{eg}
\label{eg:diagram}
We now use \cref{thm:H1tauG} and \cref{eg:H1Tsc} to compute $\Ho^1_\tau(\Gamma,G)$ for simply connected groups $G$ in which the order of $\Gamma$ equals the order of $\tau$. We summarize the result in the table below.

\begin{center}\renewcommand{\arraystretch}{1.6}
\begin{tabular}{|c|c|c|}
     \hline Order of $\tau$ & Type & $\Ho^1_\tau(\Gamma,G)$ \\ \hline
     Order 1 & Any & $\{[0]\}$ \\ \hline
    \multirow{4}{*}{Order 2} & $A_{2\ell -1}$  & $\{[0], [\calpha_\ell]\}$ \\ \cline{2-3}
    &$A_{2\ell}$ & $\{[0]\}$\\ \cline{2-3}
    &$D_{\ell +1}$ & $\{[0], [\calpha_\ell]\}$\\\cline{2-3}
    &$E_6$ & $\{[0], [\calpha_2]\}$ \\  \hline
    Order 3 &$D_4$ & $\{ [0], [\calpha_2]\}$\\ \hline
\end{tabular}\end{center}

When $\tau$ is non trivial, and when $G$ is of type $A_{2\ell}$ there is nothing to show since $\Ho^1_\tau(\Gamma,T)$ is already trivial. In type $A_{2\ell-1}$ the Weyl group $W^\tau$ acts trivially on the invariant coroots, so $\Ho^1_\tau(\Gamma,G)=\Ho^1_\tau(\Gamma,T)$. In the other cases with $\tau$ of order $2$, the Weyl group $W^\tau$ permutes the non zero invariant coroots and so we always have two elements in $\Ho^1_\tau(\Gamma,G)$, one represented by $0$ and the other one by the class of any of the invariant simple coroots. Finally, for $D_4$ and $\tau$ of order $3$, we have that $\sigma_2(\calpha_2)=-\calpha_2$ and so the two non trivial elements of $\Ho^1_\tau(\Gamma,T)$ are identified in $\Ho^1_\tau(\Gamma,G)$, which then has $2$ elements.
\end{eg}

\begin{eg} \label{eg:tautrivial} Let $\Gamma=\ZZ/m\ZZ$ and assume that $\tau$ is the trivial diagram automorphism. Then $\Ho^1_\tau(\Gamma,G)$ is identified with the quotient of $T_m$ (the $m$-th torsion of the torus $T$) by the Weyl group $W$ of $G$.
For instance, when $G=\SL_n$ and $m=2$, then \[\Ho^1_\tau(\Gamma,\SL_n) = \left\lceil \frac{n+1}{2}\right\rceil .
\] Similarly, for general $m$ and for $G=\SL_2$ we have
\[\Ho^1_\tau(\Gamma,\SL_2)= \left\lceil \frac{m+1}{2}\right\rceil.\]
\end{eg}

It was proved by Adams and Ta\"{\i}bi \cite{Adams-Taibi:2018} that the Galois cohomology of real groups is isomorphic to $\Ho^1_\theta(\Gamma, G_{\mathbb{C}})$ for the holomorphic involutions arising from Cartan involutions, and they determined the cardinality of those $\Ho^1_\theta(\Gamma, G_{\mathbb{C}})$. In \cite{Borovoi-Timashev:2021}, Borovoi and Timashev computed the Galois group of real groups in terms of Kac labels when $m=2$. This is similar to what we will do for general $m$ in what follows.

\subsubsection{Description through alcove} In this section we further describe $\Ho^1_\tau(\Gamma,G)$ as representing equivalence classes of rational points in the fundamental alcove for $\coX(T)^\tau$. We first begin with some notation that we use throughout. 

Let $\tau$ be a diagram automorphism on $G$ of order $r$ (possibly trivial). Let $G^\tau$ be the fixed subgroup of $\tau$ in $G$. Then, $G^\tau$ is connected and simple. 
Following \cite[p.8]{Hong-Kumar:2019}, set 
\begin{equation*}
  \theta_0=\begin{cases}
  \text{highest root of } G &\text{ if } r=1;\\
  \text{highest short root of }G^\tau   &\text{ if $r>1$ and } G\not= A_{2\ell}  ;\\
  \text{2 $\cdot$ highest short root of } G^\tau    &\text{ if $r=2$ and $G=A_{2\ell}$ }.  \end{cases}
\end{equation*}
Let $\check{\theta}_0$ be the dual of $\theta_0$, in other words:
\begin{equation*}
  \check{\theta}_0=\begin{cases} 
  \text{highest short  coroot of } G &\text{ if } r=1;\\
  \text{highest  coroot of }G^\tau   & \text{ if $r>1$ and } G\not= A_{2\ell}  ;\\
  \frac{1}{2} \cdot \text{ highest coroot of } G^\tau    & \text{ if $r=2$ and $G=A_{2\ell}$ }.
  \end{cases}
\end{equation*}
Let $\check{M}$ be the lattice spanned by $W^\tau\cdot \check{\theta}_0$. Then
\begin{equation*}
 \check{M}=\begin{cases} 
 \text{coroot lattice of $G$} &\text{ if } r=1;\\
  \text{lattice of long coroots of $G^\tau$}   &\text{ if $r>1$ and } G\not= A_{2\ell}  ;\\
  2\cdot \text{lattice of long coroots of $G^\tau$ }     & \text{ if $r=2$ and $G=A_{2\ell}$ }.
  \end{cases}
\end{equation*}

We consider the affine Weyl group $W^\tau_{a}:= \check{M} \rtimes W^\tau$, which naturally acts on $X_*(T)^\tau_{\mathbb{R}}$. Then the quotient $X_*(T)^\tau_{\mathbb{R}}/W^\tau_{a}$ is in bijection with the fundamental alcove   
\[ \Ac(G^\tau)=\left\lbrace \lambda \in \coX(T)^\tau_\RR \text{ such that } ( \lambda, \beta_i ) \geq 0 \text{ for } 1 \leq i \leq \ell \text{ and } ( \lambda, \theta_0 ) \leq 1 \right\rbrace,
\]
where $\beta_1, \dots, \beta_\ell$ are the simple roots of $G^\tau$.
In what follows we will use the notation $\Arm(G^\tau)$, or simply $\Arm$, to denote the finite set $\Ac(G^\tau) \cap   \frac{r}{m} X_*(T)^\tau$, for non zero integers $r$ and $m$.

\begin{lemma}\label{lemma:alcove} Let $\Gamma =\langle \gamma \rangle$ be the cyclic group of order $m$ and assume that $\gamma$ acts on the simple algebraic group $G$ via a diagram automorphism $\tau$ of order $r$. Assume that $\car(k)$ does not divide $m$. Then there is a surjective map
\[ \phi \colon \Arm \to \Ho^1_\tau(\Gamma, G)\] which is an isomorphism when $G$ is simply connected. 
\end{lemma}

\begin{proof} We first of all note that
\[\Arm \cong  \dfrac{\frac{r}{m}\coX(T)^\tau}{\check{M}\rtimes W^\tau} ,\] so that, in view of \cref{thm:H1tauG}, it will be enough to show that lattice $\norm_\tau(\coX(T))$ contains $\check{M}$. Let $T_{sc}$ denote the maximal torus of the simply connected cover $G_{sc}$ of $G$ preseved by $\tau$. Since $\coX(T_{sc})$ equals the coroot lattice and $\coX(T_{sc}^\tau)=\coX(T_{sc})^\tau$, one can directly show that $\norm_\tau(\coX(T_{sc}))$ is naturally identified with $\check{M}$ (e.g. using the computations of \cref{eg:H1Tsc}).
Since $\norm_\tau(\coX(T_{sc})) \subset \norm_\tau(\coX(T))$, the above argument shows that $\check{M}$ is contained in $\norm_\tau(\coX(T))$.
\end{proof}

\begin{eg} Let $G$ be a group of type $A_3$ and $\tau$ be its non trivial diagram automorphism of order $2$. Let $m=2k$ be the cardinality of the cyclic group $\Gamma$. From \cref{lemma:alcove} it follows that $\Ho^1_\tau(\Gamma,G)$ is in bijection with $\Ac_{2/2k}$ when $G$ is simply connected, i.e. for $G=\SL_4$. 
The group $G^\tau$ is of type $C_2$, with simple roots $\{ \alpha, \beta\}$ and with highest shortest root $\theta_0=\alpha + \beta$. 
If the group is simply connected, then $X_*(T^\tau)$ is spanned by the coroots $\calpha=\frac{\alpha}{2}$ and $\cbeta=\beta$. Then, an element $\frac{A}{k} \calpha + \frac{B}{k} \beta  \in \frac{1}{k} \coX(T^\tau)$ belongs to $\Ac_{2/2k}$ if and only if
\[ 2A -B \geq 0 \qquad -2A+2B \geq 0 \qquad  B \leq k.\]
This implies that $A$ and $B$ are necessarily integers between $0$ and $k$. We can therefore compute that
\[|\Ho^1_\tau(\Gamma,\SL_4)| = | \Ac_{2/2k}| = \begin{cases}
(\ell + 1)^2 & k = 2 \ell \\
(\ell +1)(\ell +2) & k = 2 \ell +1.
\end{cases}
\]
One can similarly compute the cardinality of $\Ho^1_\tau(\Gamma,\SL_6)$, where $\Gamma$ is the cyclic group of order $2k$ and $\tau$ acts on $\SL_6$ via the non trivial diagram automorphism. One obtains
\[|\Ho^1_\tau(\Gamma,\SL_6)| = \begin{cases}
\dfrac{\ell(\ell + 1)(\ell +2)}{3} & k = 2 \ell \\
\dfrac{(\ell +1)(\ell +2)(\ell+3)}{3} & k = 2 \ell +1
\end{cases}.
\] 
\end{eg}

\medskip We have seen in the example above how  \cref{lemma:alcove} allows us to translate the problem of computing $\Ho^1_\tau(\Gamma,G)$ into counting the number of integer solutions of a system of inequalities. These inequalities are explicitly found using the Cartan matrix of the group $G^\tau$. 

\smallskip  We use this point of view to describe $\Ho^1_\tau(\Gamma, G)$ for $G$ of adjoint type. We first set up some notation. If the group $G$ is of type $X_n$ and $\tau$ has order $r$, then the coefficients $a_i$ of $\theta_0=\sum_{i=1}^\ell a_i \beta_i$ for the group $G^\tau$ are the Kac's labels at the $i$-th vertex of the (twisted) affine Dynkin diagram of type $X_n^{(r)}$ (see \cite[Tables Aff 1, Aff 2 and Aff 3,  at pages 54 and 55]{kac:1990:infinite}), except when $G$ is of type $A_{2\ell}$ and $\tau$ has order $2$. In this latter case, $\theta_0= \sum_{i=1}^\ell 2 \beta_i$, or equivalently $a_i$ corresponds to Kac's label of $A_{2\ell}^{(2)}$ at the $(i-1)$-st vertex.

Using this notation, we define $\Srm(G^\tau)$, or simply denoted $\Srm$, to be the set whose elements are an $(\ell+1)$-tuple $(s_i)_{i=0,1,\cdots, \ell}$ of integers, satisfying 
\[ s_i\geq 0 \quad \text{for all $i\in\{0,\dots, \ell\}$,} \quad \text{ and } \quad    \sum_{i=0}^\ell a_i s_i= \frac{m}{r},\]
where $a_0=1$ (this is again Kac's label at the $0$-th vertex of the (twisted) affine Dynkin diagram, except in case $A_{2\ell}^{(2)}$, where it is Kac's label at the $\ell$-th vertex). 

\begin{lemma} \label{lemma:ArmSrm} Let $\Gamma =\langle \gamma \rangle$ be the cyclic group of order $m$ and assume that $\gamma$ acts on the simple algebraic group $G$ via a diagram automorphism $\tau$ of order $r$. Assume that $\car(k)$ does not divide $m$. 
If $G$ is of adjoint type, then $\Arm \cong \Srm$.
\end{lemma}
\begin{proof} Since $G$ is of adjoint type, then $\coX(T^\tau)$ is spanned by the coweights $\cvarpi_1,\dots, \cvarpi_\ell$, so every element $\lambda \in \Arm$ can be written as $\frac{r}{m}\sum_{i=1}^\ell s_i\cvarpi_i$ for some integers $s_i$. From $( \lambda, \beta_i ) \geq 0$ we deduce that $s_i \geq 0$ for all $i \in \{1,\dots, \ell\}$. Moreover, $( \lambda, \theta ) \leq 1$ is equivalent to $\sum_{i=1}^\ell s_ia_i \leq \frac{m}{r}$, that is $s_0:=\frac{m}{r}-\sum_{i=1}^\ell s_i a_i \geq 0$. 
\end{proof}

We introduce the equivalence relation $\sim$ on $\Srm$ which is induced by the outer automorphisms  of the affine Dynkin diagrams (possibly twisted). For example, if the affine Dynkin diagram has no outer symmetry, then the equivalence relation is trivial. If, for instance,  $G$ is of type $A_{2\ell-1}$ and the order $r$ of $\sigma$ is $2$, the equivalence relation $\sim$ is given by  
\[ (s_i) \sim (s_i') \qquad \text{ if and only if } \qquad  s_0=s'_1, \quad s_1=s_0', \text{ and }s_i=s_i' \text{ for }i \geq 2.\]
Let $\Srm/\!\sim$ be the set of the equivalence classes of $\sim$. 

\begin{prop}\label{prop:alcove}
Under the assumptions of \cref{lemma:ArmSrm}, there is a natural bijection \[\Ho_\tau^1(\Gamma, G)\cong \Srm/\!\sim.\]
\end{prop}
\begin{proof}
By \cref{lemma:ArmSrm}, there is a natural bijection 
\begin{equation}\label{eq_pre_bij}
\dfrac{\frac{r}{m}\coX(T)^\tau}{\check{M}\rtimes W^\tau} \cong \Srm.     \end{equation}
We will split the proof in three cases.

\textit{Case I.} When $\tau$ is trivial, equivalently $r=1$, we have that $\check{M}$ is the coroot lattice $\check{Q}$ of $G$, so that the left hand side of \eqref{eq_pre_bij} is $\frac{1}{m}X_*(T) / W_{\rm aff} $, where $W_{\rm aff}$ is the affine Weyl group $\check{Q}\rtimes W$. Moreover, since $G$ is adjoint, $X_*(T)$ equals to the coweight lattice $\coP$ of $G$.
Let $\Omega_a$ be the normalizer of the set of all affine simple reflections of $W_{\rm aff}$. It is known that the extended affine Weyl group $\coP\rtimes W$ is isomorphic to the semidirect product $ W_{\rm aff}\rtimes \Omega_a$ \cite[\S 2.3]{Bourbaki:1968}. In particular, $\Omega_a$ is the group of all outer automorphisms of $W_{\rm aff}$, and the map \eqref{eq_pre_bij} is $\Omega_a$-equivariant. 

Observe that we have the following bijections:
\[ \Ho_\tau^1(\Gamma,G)\cong  \dfrac{ \frac{r}{m}X_*(T)}{(\coP\rtimes W )} \cong  \left( \dfrac{\frac{r}{m}X_*(T)}{ W_{\rm aff}} \right) / (\coP/\check{Q}).   \] 
The action of $\coP/\check{Q}$ on $\frac{r}{m}X_*(T)  /W_{\rm aff} $ is equivalent to the action of $\Omega_a$, via the natural isomorphism $\Omega_a \cong \coP/\check{Q}$ \cite[\S 2.3]{Bourbaki:1968}. It follows that, there is a natural bijection $ \Ho_\tau^1(\Gamma,G)\cong \Srm/\!\sim$. 

\smallskip 

\textit{Case II.} When $(G,r)=(A_{2\ell}, 2 ), (E_6, 2), (D_4,3)$, there are no outer symmetry in the associated affine Dynkin diagram $A_{2\ell}^{(2)}, E_6^{(2)}, D_4^{(3)}$. Moreover, a direct computation allows us to compare $\check{M}$ with $\norm_\tau(\coX(T))$:
\begin{center}\renewcommand{\arraystretch}{1.6}
\begin{tabular}{|c|c|c|}
     \hline $(G,r)$ & $\check{M}$ & $\norm_\tau(\coX(T))$ \\ \hline
    $(A_{2\ell},2)$  & $\left\langle \begin{array}{c} 2\varpi_1-\varpi_2, (-\cvarpi_{i-1} +2 \cvarpi_{i} - \cvarpi_{i+1})_{i=2,\dots, \ell-2},\\ -\cvarpi_{\ell-2}+2\cvarpi_{\ell-1}-2\cvarpi_{\ell}, -\cvarpi_{\ell-1}+2\cvarpi_\ell \end{array} \right\rangle$ & $\langle (\cvarpi_i)_{i=1, \dots, \ell-1}, 2 \cvarpi_\ell \rangle$ \\ \cline{2-3}
     $(E_{6},2)$ & $\left\langle \begin{array}{c} 4 \cvarpi_1 - 2 \cvarpi_2, 2\cvarpi_1-4\cvarpi_2+4\cvarpi_3\\
     2\cvarpi_2-2\cvarpi_3+\cvarpi_4, \cvarpi_3-2\cvarpi_4 \end{array} \right\rangle $ & $\langle 2\cvarpi_1, 2\cvarpi_2, \cvarpi_3, \cvarpi_4 \rangle$ \\ \cline{2-3}
 \hline
   $(D_4,3)$ & $\langle 2\cvarpi_1-3\cvarpi_2, 3\cvarpi_1-6\cvarpi_2\rangle$ & $\langle \cvarpi_1, 3\cvarpi_3\rangle$\\ \hline
\end{tabular}\end{center}
This shows that $\norm_\tau(\coX(T)) =\check{M}$. This also completes the proof the theorem in these cases. 

\smallskip 

\textit{Case III.} We are left to consider other two cases, i.e. when $(G, r)$ is $(A_{2\ell-1}, 2)$, or $(D_{\ell+1}, 2)$. Denote by  $X_*(T)_\tau$ the quotient of $X_*(T)$ by the operator $1-\tau$. Note that there is a natural map $X_*(T)_\tau\to X_*(T)^\tau$ given by $\bar{\lambda}\mapsto \sum_{i=0}^{r-1}\tau^i(\lambda)$, where $\lambda$ is any lift $\bar{\lambda}$ in $X_*(T)$. This gives rise to the following bijections (in these two cases):
\[ X_*(T)_\tau\cong  \norm_\tau(X_*(T)), \quad \check{Q}_\tau\cong \norm_\tau(\check{Q})=\check{M} .   \]
It induces the following isomorphisms: 
\[ \norm_\tau(X_*(T))/ \check{M} \cong X_*(T)_\tau/  \check{Q}_\tau  \cong  (X_*(T)/ \check{Q})_\tau  \]
of abelian groups. By the calculation in \cite[Table 2.11]{Besson-Hong:2020}, $(X_*(T)/ \check{Q})_\tau \cong \mathbb{Z}/2\mathbb{Z}$. Thus, $\norm_\tau(X_*(T))/ \check{M}\cong \mathbb{Z}/2\mathbb{Z}$. By the argument similar to the case when $r=1$, the action of $\norm_\tau(X_*(T))/ \check{M}$ on $\frac{r}{m}X_*(T)^\tau  / \check{M}$ exactly corresponds to the only nontrivial outer automorphisms on the twisted affine Dynkin diagrams of types $A_{2\ell-1}^{(2)}$ and $ D_{\ell+1}^{(2)}$. This concludes the proof. 
\end{proof}

\subsubsection{Revisit Kac's classification of finite order automorphisms}
Let $G$ be a simple algebraic group over $k$ of characteristic $p$ and let $\sigma$ be an automorphism on $G$ such that $\sigma^m=0$ for some positive integer $m$. 
When $p=0$, it is a classical result of Kac \cite[Theorem 8.6]{kac:1990:infinite} that $\sigma$, up to conjugation, is classified by elements in $\Srm/\sim$. In fact, from our computation of $\Ho_\tau^1(\Gamma,G)$, the same classification holds for general characteristic $p$ such that $p$ does not divide $m$. 

Given a diagram automorphism $\tau$ on $G$, we consider the set 
\[ \Aut(G)_{\tau,m}:=\{ \sigma\in \Aut(G)  \,|\,  \sigma^m=1,   \bar{\sigma}=\bar{\tau} \text{ in } \Out(G)      \} ,       \]
where $\Out(G):=\Aut(G)/G_{\ad}$ is the group of outer automorphisms, and $\bar{\sigma}$ and $\bar{\tau}$ are the images of $\sigma$ and $\tau$ in $\Out(G)$. Let $\Aut(G)_{\tau,m}/\!\sim$ be the conjugation classes of $\Aut(G)_{\tau,m}$ by $G_{\ad}$. Then, 
there is a natural bijection 
\begin{equation} \label{eq:kac}  \Ho_\tau^1(\Gamma, G_{\ad})\cong   \Aut(G)_{\tau,m}/\sim\end{equation}
given by $[g]\mapsto \Ad_{{g}}\circ \tau$. Having described the left hand side of \eqref{eq:kac} in \cref{prop:alcove} in terms of $\Srm/\!\sim$, we obtain Kac's classification over a field of possibly positive characteristic.

\begin{theorem}\label{thm:clas} Let $G$ be a simple group over an algebraically closed field of characteristic $p \geq 0$. Let $\tau$ be a diagram automorphism of $G$ whose order divides $m$. If $p$ does not divide $m$, then there is a natural bijection $\Aut(G)_{\tau,m}/\! \sim \;  \cong \,  \Srm/\!\sim $.
\end{theorem}

\subsection{From diagram automorphisms to general automorphisms} \label{sec:cohomology_general} In the previous sections, we have seen how to compute $\Ho^1(\Gamma,G)$ under the assumption that the $\Gamma$ acts via diagram automorphisms only. We explore here what happens when we drop this assumption. 

For this purpose, in this section we denote by $m$ the order of the cyclic group $\Gamma$ and assume that $\car(k)$ does not divide $m$. We use the letters $\sigma$ or $\tau$ to denote automorphisms of $G$ such that $\sigma^m=\tau^m=1$. As before $\Ho^1_{\sigma}(\Gamma,G)$ (and  $\Ho^1_{\sigma}(\Gamma,G)$) denotes the non abelian cohomology of $\Gamma$ with coefficients in $G$, where we assume that the generator $\gamma$ of $\Gamma$ acts by $\sigma$ (resp. $\tau$).

We first of all note that, from the very definition of non-abelian group cohomology, we obtain the following result.

\begin{lemma}
Assume that $\sigma_2=\Ad_h\circ \sigma_1 \circ \Ad_{h^{-1}}$ for some $h\in G$.  Then, the map $\Ad_h \colon \Ho^1_{\sigma_1}(\Gamma, G)\to \Ho^1_{\sigma_2}(\Gamma, G)$ is a bijection. 
\end{lemma}

We further observe the following result.

\begin{lemma} \label{lem:H1} Let $\sigma$ and $\tau$ be automorphisms of $G$ as above and such that  $\sigma =\Ad_g \circ \tau$ for some $g \in G^{\tau}$. The multiplication map $- \cdot g \colon G \to G$ induces a bijection $\xymatrix{G/\!\sim_{\sigma} \ar[r]^-{\cong} & G/\!\sim_{\tau}}$. This induces an isomorphism $\xymatrix{\Ho^1_{\sigma}(\Gamma,G)  \ar[r]^-{\cong} & \Ho^1_{\tau}(\Gamma,G)}$ if and only if $g^{m}=1$.
\end{lemma}

\begin{proof} We are only left to show the last assertion. By definition, we have that $a \in \text{Z}^1_\sigma(\Gamma,G)$ is equivalent to $a \sigma(a) \dots \sigma^{m-1}(a)=1$. This amounts to
\begin{align*}1 = a \, g \tau(a) g^{-1} \,  g^2 \tau^2(a) g^{-2} \,  \dots \, g^{m-1}\tau^{m-1}(a) g^{m-1}= (a g) \, \tau(at)\,  \tau^2(at)\, \dots \, \tau^{m-1}(at) g^{m},\end{align*}
so that $ag\in \text{Z}^1_\tau(\Gamma,G)$ if and only if $g^m=1$, as claimed.
\end{proof}

\begin{cor}\label{cor:H1}
Let $G$ be an adjoint group, and let $\sigma_1$ and $\sigma_2$ be two automorphisms of $G$. If   $\bar{\sigma}_1=\bar{\sigma}_2 \in \text{Out}(G)$, then $\Ho^1_{\sigma_1}(\Gamma,G)=\Ho^1_{\sigma_2}(\Gamma,G)$.
\end{cor}

\begin{rmk} We have seen in \cref{prop:decom} that, under slightly more restrictive assumptions on $\car(k)$, every finite automorphism $\sigma$ of $G$ admits a decomposition as $\Ad_{t} \circ \tau$ for some diagram automorphism $\tau$ and $t \in T^\tau$. Moreover we can ensure that $t^m=1$ if and only if $|Z(G^{\tau,0})|=1$. This happens when $(G,|\tau|)$ are of type $(A_{2\ell},2)$, $(E_6,2)$, or $(D_4,3)$. In these cases then \cref{lem:H1} implies that $\Ho^1_\sigma(\Gamma,G) = \Ho^1_\tau(\Gamma,G)$ whenever $\bar{\sigma}=\bar{\tau} \in \Out(G)$. \end{rmk} 

\begin{eg}
Let $\sigma$ be involution on $\SL_n$ given by $\sigma(g)=(g^{t})^{-1}$, i.e. the one sending a matrix to the inverse of its transpose. Take $\Gamma=\mathbb{Z}/2\mathbb{Z}$. By linear algebra, $\Ho^1_\sigma(\Gamma, \SL_n)$ consists of one element. It is well-known that when $n=2i+1 $ with $i\geq 1$, then $\sigma$ is a diagram automorphism preserving a Borel subgroup (different from the group of upper-triangular matrices).  We can see that this calculation agrees with the calculation in \cref{eg:diagram}.  When $n=2i$ with $i\geq 1$, the automorphism $\sigma$ is not a diagram automorphism and in this case, the cardinality of  $\Ho_\sigma^1(\Gamma, \SL_n)$ does not coincide with the cardinality of $\Ho_\tau^1(\Gamma, \SL_n)$, where $\tau$ is the diagram automorphism of $\SL_n$ of order $2$ for $n \geq 3$, and the trivial diagram automorphism for $\SL_2$.

Let $G$ be the group $\PGL_n$ and let $\sigma$ be the involution as above. Let $g$ be an invertible matrix in $\GL_n$ and denote by $\bar{g}$ is equivalence class in $\PGL_n$. The condition $\bar{g}\sigma(\bar{g})=1$ holds if and only if  there exists a non zero scalar $\lambda \in k$ such that $g^t=\lambda g$. This implies that $\lambda^2=1$, thus $\lambda=\pm 1$. If $n=2i+1$, $\lambda=1$, as there is no invertible anti-symmetric matrix. It follows that when $n=2i+1$, the set $\Ho_\sigma^1(\Gamma, \PGL_n)$ has one element. When $n=2i$, instead, $g$ can be either symmetric or anti-symmetric matrices and this determines two distinct elements of $\Ho^1_\sigma(\Gamma, \PGL_n)$.  Note that the cardinalities of these cohomology spaces agree with the calculation in \cref{eg:diagram} and with  \cref{cor:H1}.  
\end{eg}

\bibliographystyle{alpha}
\bibliography{BiblioFinal}
\vfill 

\end{document}